\documentclass[12pt,a4paper]{article}

\usepackage{amsmath,amssymb}
\usepackage{latexsym}
\usepackage{amsmath}
\usepackage{amssymb}
\usepackage{amsthm}
\usepackage{amscd}
\usepackage{mathrsfs}
\usepackage[all]{xy}
\usepackage{graphicx}

\usepackage{bm}
\usepackage{comment}

\newtheorem{theorem}{Theorem}[section]
\newtheorem{lemma}[theorem]{Lemma}
\newtheorem{corollary}[theorem]{Corollary}
\newtheorem{proposition}[theorem]{Proposition}

\theoremstyle{definition}

\newtheorem{remark}[theorem]{Remark}
\newtheorem{definition}[theorem]{Definition}
\newtheorem{example}[theorem]{Example}

\theoremstyle{remark}

\makeatletter

\@addtoreset{equation}{section}
\makeatother

\DeclareMathOperator{\Spec}{Spec}

\DeclareMathOperator{\Gal}{Gal}

\DeclareMathOperator{\Ker}{Ker}

\DeclareMathOperator{\Aut}{Aut} 
 
\DeclareMathOperator{\Ind}{Ind} 
\DeclareMathOperator{\Tr}{Tr}

\DeclareMathOperator{\Nr}{Nr}

\DeclareMathOperator{\ind}{\mathrm{Ind}}

\makeatletter
\def\widebreve#1{\mathop{\vbox{\m@th\ialign{##\crcr\noalign{\kern\p@}%
  \brevefill\crcr\noalign{\kern0.1\p@\nointerlineskip}%
  $\hfil\displaystyle{#1}\hfil$\crcr}}}\limits}

\def\brevefill{$\m@th \setbox\z@\hbox{}%
 \hfill\scalebox{0.7}{\rotatebox[origin=c]{90}{(}} \kern4pt $}
\makeatletter

\begin{document}
\title{Local Galois representations associated to\\ additive polynomials}
\author{Takahiro Tsushima}
\date{}
\maketitle
 \begin{abstract}
 For an additive polynomial and a positive integer, 
 we define an irreducible smooth representation of a Weil group of a non-archimedean local field. We study several invariants of this representation. 
We deduce a necessary and sufficient condition for it to be primitive. 
 \end{abstract}
 \section{Introduction}
Let $p$ be a prime number and $q$ a power of it. 
An additive polynomial $R(x)$ over $\mathbb{F}_q$ is a one-variable polynomial 
with coefficients in $\mathbb{F}_q$ such that 
$R(x+y)=R(x)+R(y)$.  
It is known that $R(x)$ has the form
$\sum_{i=0}^e a_i x^{p^i}\ (a_e \neq 0)$ with 
an integer $e \geq 0$. 
Let $F$ be a non-archimedean local field with residue field $\mathbb{F}_q$. 
We take a separable closure $\overline{F}$ of $F$. 
Let $W_F$ be the Weil group of $\overline{F}/F$. 
Let $v_F(\cdot)$ denote the normalized valuation on $F$. 
We take a prime number $\ell \neq p$. 
For a non-trivial character $\psi \colon \mathbb{F}_p \to \overline{\mathbb{Q}}_{\ell}^{\times}$, a non-zero additive polynomial $R(x)$ over $\mathbb{F}_q$ and a positive integer $m$ which is prime to $p$, 
we define an irreducible smooth $W_F$-representation
$\tau_{\psi,R,m}$ over $\overline{\mathbb{Q}}_{\ell}$ of degree $p^e$ 
if $v_F(p)$ is sufficiently large. 
This is unconditional if $F$ has positive characteristic. 
The integer $m$ is related to the Swan conductor exponent of $\tau_{\psi,R,m}$. 
As $m$ varies, the isomorphism class of 
 $\tau_{\psi,R,m}$ varies. 

Let $C_R$ denote the 
algebraic affine curve defined by $a^p-a=x R(x)$ 
in $\mathbb{A}_{\mathbb{F}_q}^2
=\Spec \mathbb{F}_q[a,x]$. 
This curve is studied in \cite{GV} and 
\cite{BHMSSV} in detail.
For example, the smooth compactification of $C_R$ is proved to be supersingular if $(p,e) \neq (2,0)$.  
The automorphism group of $C_R$ contains a semidirect product $Q_R$ of a cyclic group and an  extra-special $p$-group (\eqref{fa0}).   
Let $\mathbb{F}$ be an algebraic closure of $\mathbb{F}_q$. Then a semidirect group $Q_R \rtimes \mathbb{Z}$ acts on 
the base change $C_{R,\mathbb{F}}:=C_R \times_{\mathbb{F}_q} \mathbb{F}$ as endomorphisms, where $1 \in \mathbb{Z}$ acts on $C_{R,\mathbb{F}}$ as the Frobenius endomorphism over $\mathbb{F}_q$. 
The center $Z(Q_R)$ of $Q_R$ is identified with 
$\mathbb{F}_p$, which acts on $C_R$ as 
$a \mapsto a+\zeta$ for $\zeta \in \mathbb{F}_p$. 
Let $H_{\rm c}^1(C_{R,\mathbb{F}},\overline{\mathbb{Q}}_{\ell})$ be  
the first \'etale cohomology group of $C_{R,\mathbb{F}}$ with compact support. 
Each element of $Z(Q_R)$ is fixed by the action of $\mathbb{Z}$ on $Q_R$. 
Thus its $\psi$-isotypic part $H_{\rm c}^1(C_{R,\mathbb{F}},\overline{\mathbb{Q}}_{\ell})[\psi]$ is 
regarded as a $Q_R \rtimes \mathbb{Z}$-representation. 

We construct a concrete 
Galois extension over $F$ whose Weil group is isomorphic to a subgroup of 
$Q_R \rtimes \mathbb{Z}$
associated to the integer $m$ (Definition \ref{3d} and \eqref{weil}). 
Namely we will define a homomorphism $\Theta_{R,m} \colon W_F \to 
Q_R \rtimes \mathbb{Z}$ in \eqref{kkt}. As a result, we define 
$\tau_{\psi,R,m}$ to be the composite 
\[
W_F \xrightarrow{\Theta_{R,m}}
Q_R \rtimes \mathbb{Z} \to \mathrm{Aut}_{\overline{\mathbb{Q}}_{\ell}}(H_{\rm c}^1(C_{R,\mathbb{F}},\overline{\mathbb{Q}}_{\ell})[\psi]). 
\]
This is a smooth irreducible representation of 
$W_F$ of degree $p^e$.

 We state our motivation and reason 
 why we introduce and study $\tau_{\psi,R,m}$. 
 It is known that the reductions of concentric affinoids
 in the Lubin--Tate curve fit into 
 this type of curves $C_R$ with special $R$. For example, 
 see \cite{T} and \cite{W}.  
 When $R$ is a monomial and $m=1$, the representation $\tau_{\psi,R,m}$ is studied in \cite{IT} and \cite{IT2} in detail. 
 In these papers, 
the reduction of a certain affinoid 
 in the Lubin--Tate space is related to $C_R$
 in some sense and the supercuspidal representation $\pi$ of 
 $\mathrm{GL}_{p^e}(F)$ which corresponds to 
 $\tau_{\psi,R,m}$ under the local Langlands correspondence explicitly. The homomorphism 
 $\Theta_{R,1}$ with $R(x)=x^{p^e}\ (e \in \mathbb{Z}_{\geq 1})$ does appear in the work \cite{IT}.  
 An irreducible representation of a group is said to be primitive if it is not isomorphic to an induction of any representation of a proper subgroup.  
The representation $\tau_{\psi,R,m}$ in \cite{IT} and \cite{IT2} is primitive and 
 this property makes it complicated to describe 
 $\pi$ in a view point of 
 type theory.   
 It is an interesting problem to 
 do the same thing for general  
 $\tau_{\psi,R,m}$. 
 In this direction, it would be valuable to 
 know when $\tau_{\psi,R,m}$ is primitive. 
 We expect that another $C_R$ will be related to 
 concentric affinoids in the Lubin--Tate spaces as in \cite{IT}.  
   
We briefly explain the content of each section.    
In \S\ref{2}, we state 
several things on the curves $C_R$ and the extra-special $p$-subgroups contained in the automorphism groups of the curves. 

In \S\ref{3.1} and \S \ref{3.2}, we construct the Galois extension mentioned above 
and define $\tau_{\psi,R,m}$.  
Let $d_R:=\mathrm{gcd}\{p^i+1 \mid a_i \neq 0\}$. 
We show that the Swan conductor exponent of 
$\tau_{\psi,R,m}$ equals $m(p^e+1)/d_R$
(Corollary \ref{indd}). 
 In \S \ref{3.3}, we study primitivity of $\tau_{\psi,R,m}$. In particular, we write  down a necessary and sufficient condition for $\tau_{\psi,R,m}$ to be primitive. 
By this, we give examples such that $\tau_{\psi,R,m}$ is primitive (Example \ref{ae}).   The necessary and sufficient condition is that 
a symplectic module $(V_R,\omega_R)$ associated to $\tau_{\psi,R,m}$ is completely anisotropic (Corollary \ref{mc}). 
If $R$ is a monomial, $(V_R,\omega_R)$ is 
studied in \S \ref{root} in more detail. In Proposition \ref{pree}, 
a primary module in the sense of \cite[\S9]{K} 
is constructed geometrically by using the K\"unneth formula.  

Our aim in \S \ref{4} is to show the following 
theorem. 
 \begin{theorem}\label{main4}
 Assume $p\neq 2$. 
 The following two conditions are equivalent.
 \begin{itemize}
\item[{\rm (1)}]
There exists a non-trivial finite \'etale morphism 
\[
C_R \to C_{R_1};\ (a,x) \mapsto \left(a-\Delta(x),r(x)\right), 
\]
where $\Delta(x) \in \mathbb{F}_q[x]$ and $r(x), R_1(x)$ are additive polynomials over $\mathbb{F}_q$ such that  
$d_{R,m} \mid d_{R_1}$ and $r(\alpha x)
=\alpha r(x)$ for any $\alpha \in \mu_{d_{R,m}}$. 
 \item[{\rm (2)}] The $W_F$-representation 
 $\tau_{\psi,R,m}$ is imprimitive. 
 \end{itemize} 
 \end{theorem}
If $\tau_{\psi,R,m}$ is imprimitive, it is written as a form of 
an induced representation of a certain explicit 
$W_{F'}$-representation 
$\tau'_{\psi,R_1,m}$ 
associated to a finite extension $F'/F$. 
The proof of the above theorem depends on 
several geometric properties of $C_R$ developed in \cite{GV} and \cite{BHMSSV}. See 
the beginning of \S \ref{4} for more details. 

\subsection*{Notation}
Let $k$ be a field. Let $\mu(k)$ denote 
the set of all roots of unity in $k$.
For a positive integer $r$, let $\mu_r(k):=\{x \in k \mid x^r=1\}$.    

For a positive integer $i$, let $\mathbb{A}^i_k$ and $\mathbb{P}^i_k$
be an $i$-dimensional affine space 
and a projective space over $k$, respectively. 
For a scheme $X$ over $k$ and a field 
extension $l/k$, let $X_l$ denote the base change of $X$ to $l$. 
For a closed subset $Z$ of a variety  
$X$, we regard $Z$ as a closed subscheme 
with reduced scheme structure. 

Throughout this paper, 
we set 
$q:=p^f$ with a positive integer $f$. 
For a positive integer $i$, 
we simply write  $\Nr_{q^i/q}$ and 
$\Tr_{q^i/q}$ for the norm map and 
the trace map from $\mathbb{F}_{q^i}$ to 
$\mathbb{F}_q$, respectively.  

Let $X$ be a scheme over $\mathbb{F}_q$, 
let $F_q \colon X \to X$ be the $q$-th power 
Frobenius endomorphism. 
Let $\mathbb{F}$ be an algebraic closure of 
$ \mathbb{F}_q$. 
Let $\mathrm{Fr}_q \colon X_{\mathbb{F}} \to X_{\mathbb{F}}$ be the base change of $F_q$. 
This endomorphism $\mathrm{Fr}_q$ is called the 
Frobenius endomorphism of $X$ over $\mathbb{F}_q$. 

For a Galois extension $l/k$, let $\Gal(l/k)$
denote the Galois group of the extension.  
\section{Extra-special $p$-groups and 
affine curves}\label{2}
\begin{definition}
Let $k$ be a field.  
A polynomial $f(x) \in k[x]$ is called additive 
if $f(x+y)=f(x)+f(y)$. 
Let $\mathscr{A}_{k}$ be the set of all
additive polynomials with coefficients in $k$. 
\end{definition}
Let $p$ be a prime number. We simply 
$\mathscr{A}_q$ for 
$\mathscr{A}_{\mathbb{F}_q}$. 
Let $R(x):=\sum_{i=0}^e a_i x^{p^i} \in \mathscr{A}_q$ with $e \in \mathbb{Z}_{\geq 0}$ and 
$a_e \neq 0$.
Let 
\begin{equation}\label{1} 
E_R(x):=R(x)^{p^e}+\sum_{i=0}^e (a_i x)^{p^{e-i}} \in \mathscr{A}_q. 
\end{equation} 
We always assume 
\begin{equation}\label{always}
(p,e) \neq (2,0).
\end{equation}
We simply write  $\mu_r$
for $\mu_r(\mathbb{F})$ for a positive integer $r$.   
Let 
\[
d_R:=\mathrm{gcd}\{p^i+1 \mid a_i \neq 0 \}. 
\]
If $a_i \neq 0$, we have 
$\alpha^{p^i}=\alpha^{-1}$
and 
$\alpha^{p^{e-i}}=\alpha$ 
for $\alpha \in \mu_{d_R}$. 
Hence we have 
\begin{equation}\label{12}
\alpha R(\alpha x)= R(x), \quad 
E_R(\alpha x)=\alpha E_R(x) \quad 
\textrm{for 
$\alpha \in \mu_{d_R}$}.
\end{equation}
Let 
\[
f_R(x,y):=-\sum_{i=0}^{e-1}\left(\sum_{j=0}^{e-i-1} (a_i x^{p^i} y)^{p^j}
+(x  R(y))^{p^i}\right) \in \mathbb{F}_q[x,y]. 
\]
This is linear with respect to $x$ and $y$. 
By \eqref{12}, we have 
\begin{equation}\label{a-1}
f_R(\alpha x,\alpha y)
=f_R(x,y) \quad 
\textrm{for $\alpha \in \mu_{d_R}$}.
\end{equation} 
\begin{lemma}\label{a}
We have
$
f_R(x,y)^p-f_R(x,y)=-x^{p^e} E_R(y)+xR(y)+y R(x). 
$ In particular, if $E_R(y)=0$, 
we have $f_R(x,y)^{p}-f_R(x,y)
=xR(y)+y R(x)$. 
\end{lemma}
\begin{proof}
The former equality follows from 
\begin{align*}
f_R(x,y)^{p}-f_R(x,y)
&=xR(y)-(xR(y))^{p^e}
+\sum_{i=0}^{e-1} ( a_i x^{p^i}y-(a_i x^{p^i} y)^{p^{e-i}}) \\
&=-x^{p^e} E_R(y)+xR(y)+y R(x). 
\end{align*}
\end{proof}
\begin{definition}\label{qdef}
\begin{itemize}
\item[{\rm (1)}]
Let 
$
V_R:=\{\beta \in \mathbb{F} \mid E_R(\beta)=0\}, 
$
 which is a $(2e)$-dimensional $\mathbb{F}_p$-vector space. 
\item[{\rm (2)}]
Let 
\[
Q_R:=\left\{
(\alpha,\beta,\gamma) \in \mathbb{F}^3 \mid 
\alpha \in \mu_{d_R}, \ 
\beta \in V_R,\ \gamma^{p}-\gamma=\beta R(\beta)
\right\} 
\]
be the group whose group law is given by 
\[ 
(\alpha_1,\beta_1,\gamma_1)  \cdot 
(\alpha_2,\beta_2,\gamma_2):=\left(\alpha_1\alpha_2,\beta_1+\alpha_1 \beta_2,\gamma_1+\gamma_2+f_R(\beta_1,\alpha_1 \beta_2)\right). 
\] 
This is well-defined according to \eqref{12} and Lemma \ref{a}. 
\item[{\rm (3)}] 
Let 
$H_R:=\{(\alpha,\beta,\gamma)\in Q_R \mid \alpha=1\}, $ 
which is a normal subgroup of $Q_R$.  
\end{itemize}
\end{definition}
If $e=0$, we have $p \neq 2$ by \eqref{always}. 
We have $H_R=\mathbb{F}_p \subset 
Q_R=\mu_2 \times \mathbb{F}_p$ if $e=0$.

For a group $G$ and elements 
$g,g' \in G$, let 
$[g,g']:=gg'g^{-1}g'^{-1}$. 
\begin{lemma}\label{h1}
For $g=(1,\beta,\gamma),\ g'=(1,\beta',\gamma') \in H_R$, we have $[g,g']=(1,0,f_R(\beta,\beta')-f_R(\beta',\beta))$. 
In particular, we have $f_R(\beta,\beta')-f_R(\beta',\beta) \in \mathbb{F}_p$. 
\end{lemma}
\begin{proof}
This is directly checked. We omit the details. 
\end{proof}
For a group $G$, let $Z(G)$ denote its center and $[G,G]$ the commutator subgroup of 
$G$. 
\begin{definition}
A non-abelian $p$-group $G$ is called an \textit{extra-special $p$-group} if $[G,G]=Z(G)$ and 
$|Z(G)|=p$. 
\end{definition}
\begin{lemma}\label{ab}
Assume $e \geq 1$. 
\begin{itemize}
\item[{\rm (1)}] 
The group $H_R$ is non-abelian. 
We have $Z(H_R)=Z(Q_R)= 
\{(1,0,\gamma) \mid \gamma \in \mathbb{F}_{p}\}$. 
The quotient $H_R/Z(H_R)$ is 
isomorphic to $V_R$. 
\item[{\rm (2)}] 
The group $H_R$ is an extra-special $p$-group. 
The pairing 
$\omega_R \colon V_R \times V_R \to \mathbb{F}_p;\ 
(\beta,\beta') \mapsto f_R(\beta,\beta')-f_R(\beta',\beta)$ is a non-degenerate symplectic form. 
\end{itemize}
\end{lemma}
\begin{proof}
We show (1).  
Let  
$X_{\beta}:=\{x \in \mathbb{F} \mid f_R(\beta,x)=f_R(x,\beta)\}$
for $\beta \in V_R$. 
Then $X_{\beta}$ is an $\mathbb{F}_p$-vector space of dimension $2e-1$ if $\beta \neq 0$. 
Since $V_R$ has dimension $2e$, 
 we have $V_R \nsubseteq X_{\beta}$ for $\beta \in V_R \setminus \{0\}$. This implies 
 that $H_R$ is non-abelian according to Lemma \ref{h1}. 
 
Clearly we have $Z:=\{(1,0,\gamma) \mid \gamma \in \mathbb{F}_{p}\} \subset Z(Q_R) \subset Z(H_R)$. It suffices to show $Z(H_R) \subset Z$. 
 Let $(1,\beta,\gamma) \in Z(H_R)$. 
 We have $V_R \subset  X_{\beta}$ by Lemma \ref{h1}. This implies $\beta=0$. 
Hence we obtain $Z(H_R)
\subset Z$. The last claim is easily verified.

We show (2). 
By 
Lemma \ref{h1}, we have 
$[H_R,H_R] \subset Z(H_R)$. 
Since $H_R$ is non-abelian, 
$[H_R,H_R]$ is non-trivial.  
Hence we have $[H_R,H_R]= Z(H_R)$ by $|Z(H_R)|=p$.  Hence $H_R$ is 
extra-special. 
Assume $\omega_R(\beta,\beta')=0$ for any 
$\beta' \in V_R$. We take an element
$(1,\beta,\gamma)
\in H_R$. 
By Lemma \ref{h1}, we have $(1,\beta,\gamma) \in Z(H_R)$. Hence we have $\beta=0$ by (1). 
\end{proof}
\begin{definition}\label{cr}
\begin{itemize}
\item[{\rm (1)}] 
Let $C_R$ be the affine curve 
over $\mathbb{F}_q$
defined by $a^{p}-a=x R(x)$. 
\item[{\rm (2)}]
Let $Q_R$ act on $C_{R,\mathbb{F}}$ by 
\begin{equation}\label{fa0}
(a,x) \cdot (\alpha,\beta,\gamma)=
\left(a+f_R(x,\beta)+\gamma, \alpha^{-1} (x+\beta)\right),   
\end{equation}
for $(a,x) \in C_{R,\mathbb{F}}$ and $ (\alpha,\beta,\gamma) \in Q_R$. 
This is well-defined by \eqref{12} and Lemma \ref{a}. 
\end{itemize}
\end{definition}
The curve $C_R$ is studied in \cite{GV} and \cite{BHMSSV}.  

We take a prime number $\ell \neq p$. 
For a finite abelian group $A$, let 
$A^{\vee}$ denote the character group $\mathrm{Hom}_{\mathbb{Z}}(A,\overline{\mathbb{Q}}_{\ell}^{\times})$. 
For a representation $M$ of $A$
over $\overline{\mathbb{Q}}_{\ell}$
and $\chi \in A^{\vee}$, let 
$M[\chi]$ denote the $\chi$-isotypic 
part of $M$. 

According to Lemma \ref{ab}(1), 
we identify a character $\psi \in \mathbb{F}_p^{\vee}$ with a character of $Z(H_R)$.    
\begin{lemma}\label{fa}
Let $\psi \in \mathbb{F}_{p}^{\vee} \setminus \{1\}$.
\begin{itemize}
\item[{\rm (1)}]
Let $W \subset V_R$ be an $\mathbb{F}_p$-subspace of dimension $e$, which is totally isotropic with respect to $\omega_R$. 
Let $W' \subset H_R$ be the inverse image 
of $W$ by the natural map 
$H_R \to V_R;\ (1,\beta,\gamma) \mapsto \beta$. 
Let $\xi \in W'^{\vee}$ be an extension of $\psi \in 
Z(H_R)^{\vee}$. 
Let $\rho_{\psi}:=\mathrm{Ind}_{W'}^{H_R} \xi$. 
Then $\rho_{\psi}$ is a unique (up to isomorphism) 
irreducible representation of $H_R$ containing 
$\psi$. In particular, 
$\rho_{\psi}|_{Z(H_R)}$ is a multiple of $\psi$. 
\item[{\rm (2)}] 
The $\psi$-isotypic part 
$H_{\rm c}^1(C_{R,\mathbb{F}},\overline{\mathbb{Q}}_{\ell})[\psi]$ 
is isomorphic to $\rho_{\psi}$ as $H_R$-representations. 
\end{itemize}
\end{lemma}
\begin{proof}
By Lemma \ref{ab}(2) and 
\cite[16.14(2) Satz]{H}, the claim (1) follows. 
By \cite[Remark 3.29]{T}, we have 
$\dim H_{\rm c}^1(C_{R,\mathbb{F}},\overline{\mathbb{Q}}_{\ell})[\psi]=p^e$. 
Hence the claim (2) follows from (1). 
\end{proof}
The representation $\rho_{\psi}$
induces a projective 
representation $\bar{\rho} \colon H_R/Z(H_R)
\to \mathrm{PGL}_{p^e}(\overline{\mathbb{Q}}_{\ell})$. 
\begin{lemma}\label{inj}
The map $\bar{\rho}$ is injective. 
\end{lemma}
\begin{proof}
As in the proof of \cite[Theorem 6]{R}, 
we have $\Tr \rho_{\psi}(x)=0$ for 
$x \in H_R \setminus Z(H_R)$. 
Assume $\bar{\rho} (x Z(H_R))=1$ for 
$x \in H_R$. Then $\rho_{\psi}(x)$
is a non-zero scalar matrix. 
Hence $\Tr \rho_{\psi}(x) \neq 0$. 
This implies  
$x \in Z(H_R)$.   
\end{proof}

Let $\mathbb{Z} \ni 1$ act on 
$H_{\rm c}^1(C_{R,\mathbb{F}},\overline{\mathbb{Q}}_{\ell})$ by 
the pull-back $\mathrm{Fr}_q^\ast$. 
Let $\mathbb{Z} \ni 1$ act on 
$Q_R$ by $(\alpha,\beta,\gamma) \mapsto
(\alpha^{q^{-1}},\beta^{q^{-1}},\gamma^{q^{-1}})$. The semidirect product 
$Q_R \rtimes \mathbb{Z}$ acts on 
$H_{\rm c}^1(C_{R,\mathbb{F}},\overline{\mathbb{Q}}_{\ell})[\psi]$. 

 Let $\overline{C}_R$ denote the smooth compactification of $C_R$. 
 \begin{proposition}\label{super}
 The projective curve $\overline{C}_R$
 is supersingular. In particular, this curve has   positive genus. 
 The natural map 
 $H_{\rm c}^1(C_{R,\mathbb{F}},\overline{\mathbb{Q}}_{\ell}) \to H^1(\overline{C}_{R,\mathbb{F}},\overline{\mathbb{Q}}_{\ell})$ is an isomorphism. 
 \end{proposition}
 \begin{proof}
The former claim is shown in 
\cite[Theorems (9.4) and (13.7)]{GV} (
\cite[Proposition 8.5]{BHMSSV}). 
The last claim follows from \cite[Lemmas 3.27 and 3.28(3)]{T}. 
 \end{proof}
\section{Local Galois representation}\label{3}
In this section, we define an irreducible smooth 
$W_F$-representation 
$\tau_{\psi,R,m}$ and determine 
several invariants associated to it. 
In  \S\ref{322}, we determine the Swan conductor 
exponent of $\tau_{\psi,R,m}$.
In \S \ref{3.3}, we determine 
the symplectic module associated to $\tau_{\psi,R,m}$, and 
give a necessary and sufficient condition
for $\tau_{\psi,R,m}$ to be primitive. 
As a result, we obtain several 
examples such that $\tau_{\psi,R,m}$ is primitive.  
If $R$ is a monomial, 
we calculate invariants of the root system 
corresponding to $(V_R,\omega_R)$ 
defined in \cite{K} (
Lemma \ref{ggf}). 
\subsection{Galois extension}\label{3.1}
For a valued field $K$, let $\mathcal{O}_K$ denote the 
valuation ring of $K$. 

Let $F$ be a non-archimedean local field.
Let $\overline{F}$ be a separable closure of $F$. 
Let $\widehat{\overline{F}}$ denote the completion of $\overline{F}$. Let 
$v(\cdot)$ denote the unique valuation on 
$\widehat{\overline{F}}$ such that $v(\varpi)=1$
for a uniformizer $\varpi$ of $F$, which we 
now fix. 
We simply write  $\mathcal{O}$ for
$\mathcal{O}_{\widehat{\overline{F}}}$. 
Let  $\mathfrak{p}$ be the maximal ideal of 
$\mathcal{O}$. 

Let $q$ be the cardinality of 
the residue field of $\mathcal{O}_F$. 

We take $R(x)=\sum_{i=0}^e a_i x^{p^i} \in \mathscr{A}_q$. 
For an element $a \in \mathbb{F}_q$, 
let $\widetilde{a} \in \mu(F) \cup \{0\}$ be its 
Teichm\"uller lift. 
Let 
\[ 
\widetilde{R}(x):=\sum_{i=0}^e \widetilde{a}_i x^{p^i}, \quad \widetilde{E}_R(x):=\widetilde{R}(x)^{p^e}+\sum_{i=0}^e
(\widetilde{a}_i x)^{p^{e-i}} \in \mathcal{O}_F[x]. 
\] 
Similarly as \eqref{12}, we have 
\begin{equation}\label{a12}
\alpha\widetilde{R}(\alpha x)=\widetilde{R}(x), \quad 
\widetilde{E}_R(\alpha x)=\alpha \widetilde{E}_R(x)
\quad \textrm{for $\alpha \in \mu_{d_R}(\overline{F})$}. 
\end{equation}
\begin{definition}\label{3d}
Let $m$ be a positive integer prime to 
$p$.  
Let  
$\alpha_{R,\varpi}, \beta_{R,m,\varpi}, \gamma_{R,m,\varpi} \in
\overline{F}$ be elements such that 
\begin{align*}
\alpha_{R,\varpi}^{d_R}=\varpi, \quad 
\widetilde{E}_R(\beta_{R,m,\varpi})=\alpha_{R,\varpi}^{-m}, \quad 
\gamma_{R,m,\varpi}^{p}-\gamma_{R,m,\varpi}=
\beta_{R,m,\varpi} \widetilde{R}(\beta_{R,m,\varpi}).
\end{align*}
For simplicity, we write 
$\alpha_R, \beta_{R,m}, \gamma_{R,m}$
for 
$\alpha_{R,\varpi}, \beta_{R,m,\varpi}, \gamma_{R,m,\varpi}$, respectively. 
\end{definition}
\begin{remark}\label{rb}
By $\deg \widetilde{E}_R(x)=p^{2e}$
and $\deg \widetilde{R}(x)=p^e$, we have 
\[
v(\alpha_R)=\frac{1}{d_R}, \quad 
v(\beta_{R,m})=-\frac{m}{p^{2e}d_R}, \quad 
v(\gamma_{R,m})=-\frac{m(p^e+1)}{p^{2e+1}d_R}. 
\]
\end{remark}
The integer $m$ controls the depth of 
ramification of the resulting 
field extension $F(\alpha_R, \beta_{R,m}, \gamma_{R,m})/F$. We will understand this 
later in \S \ref{322}. 

Let 
\[
\widetilde{f}(x,y):=
-\sum_{i=0}^{e-1} \left(\sum_{j=0}^{e-i-1} (\widetilde{a}_j x^{p^i}y)^{p^j}+ (x \widetilde{R}(y))^{p^i}\right). 
\]
Let $\mathfrak{p}[x]:=\mathfrak{p} \mathcal{O}[x]$ and $\mathfrak{p}[x,y]:=\mathfrak{p} \mathcal{O}[x,y]$. 
We assume that 
\begin{gather}\label{assume}
\begin{aligned}
& \beta_{R,m}^{p^e}(\widetilde{E}_R(\beta_{R,m}+x)-\widetilde{E}_R(\beta_{R,m})-\widetilde{E}_R(x)), \quad \beta_{R,m} (\widetilde{R}(\beta_{R,m}+x)
-\widetilde{R}(\beta_{R,m})-\widetilde{R}(x)), \\
& \textrm{$\widetilde{f}(\beta_{R,m},x)^{p}-\widetilde{f}(\beta_{R,m},x)+\beta_{R,m}^{p^e} \widetilde{E}_R(x)-x \widetilde{R}(\beta_{R,m})-\beta_{R,m} \widetilde{R}(x)$ are contained in $\mathfrak{p}[x]$ and}\\
& (\gamma_{R,m}+\widetilde{f}(\beta_{R,m},y)+x)^{p}-\gamma_{R,m}^{p}-\widetilde{f}(\beta_{R,m},y)^{p}-x^{p} \in \mathfrak{p}[x,y]. 
\end{aligned}
\end{gather}
For $r \in \mathbb{Q}_{\geq 0}$ and 
$f,g \in \overline{F}$, 
we write $f \equiv g \mod r+$ if $v(f-g)>r$. 
For a local field $K$ contained in $\overline{F}$, 
let $W_K$ be the Weil group of $\overline{F}/K$. 
For $\sigma \in W_K$, let $n_{\sigma} \in \mathbb{Z}$ be the integer such that 
$\sigma(x)\equiv x^{q^{-n_{\sigma}}} \mod 0+$ for $x \in \mathcal{O}_{\overline{F}}$. Let $v_K(\cdot)$ denote the 
normalized valuation on $K$. 
\begin{definition}
For $\sigma \in W_F$, we set 
\begin{gather}\label{abcr}
\begin{aligned}
a_{R,\sigma}&:=\sigma(\alpha_R)/\alpha_R \in 
\mu_{d_R}(\overline{F}), \quad 
b_{R,\sigma}:=
a_{R,\sigma}^m \sigma(\beta_{R,m})-\beta_{R,m}, \\ 
c_{R,\sigma}&:=\sigma(\gamma_{R,m})-\gamma_{R,m}
-\widetilde{f}(\beta_{R,m},b_{R,\sigma}). 
\end{aligned}
\end{gather}
In the following, we simply write  
$a_{\sigma}, b_{\sigma},c_{\sigma}$ for 
$a_{R,\sigma}, b_{R,\sigma}, c_{R,\sigma}$, respectively. 
\end{definition}
For an element $x \in \mathcal{O}$, 
let $\bar{x}$ denote the image of $x$ by the map $\mathcal{O} \to \mathcal{O}/\mathfrak{p}$.   
In the following proofs, for simplicity, we often write  $\alpha$, $\beta$ and $\gamma$ for $\alpha_R$, 
$\beta_{R,m}$ and $\gamma_{R,m}$, respectively. 
\begin{lemma}\label{abbc}
We have $b_{\sigma}, c_{\sigma} \in 
\mathcal{O}$, $E_R(\bar{b}_{\sigma})=0$ and 
$\bar{c}_{\sigma}^{p}-\bar{c}_{\sigma}=\bar{b}_{\sigma}
R(\bar{b}_{\sigma})$.  
\end{lemma}
\begin{proof}
Using \eqref{a12}, the equality 
$\widetilde{E}_R(\beta)=\alpha^{-m}$
in Definition \ref{3d}
and \eqref{abcr}, 
\[
\widetilde{E}_R(\beta+b_{\sigma})
=\widetilde{E}_R(a_{\sigma}^m \sigma(\beta))
=a_{\sigma}^m \widetilde{E}_R(\sigma(\beta))
=a_{\sigma}^m \sigma(\alpha)^{-m}=\alpha^{-m}
=\widetilde{E}_R(\beta). 
\]
Using $v(\beta)<0$ in Remark \ref{rb} and \eqref{assume}, 
we have $\Delta(x):=\widetilde{E}_R(\beta+x)-\widetilde{E}_R(\beta)-\widetilde{E}_R(x) \in \mathfrak{p}[x]$. 
By letting $x=b_{\sigma}$ and applying 
the previous relationship, we deduce that 
$\widetilde{E}_R(b_{\sigma})+\Delta(b_{\sigma})=0$.  Hence $b_{\sigma} \in \mathcal{O}$ and 
$E_R(\bar{b}_{\sigma})=0$.  

By \eqref{assume}, we have 
\begin{equation}\label{ue}
\beta\widetilde{R}(\beta+b_{\sigma}) 
\equiv 
\beta \widetilde{R}(\beta)
+\beta \widetilde{R}(b_{\sigma}), \quad 
\widetilde{f}(\beta,b_{\sigma})^{p}-\widetilde{f}(\beta,b_{\sigma}) \equiv b_{\sigma}\widetilde{R}(\beta)+\beta \widetilde{R}(b_{\sigma}) \mod 0+.
\end{equation} 
Substituting $y=b_{\sigma} \in 
\mathcal{O}$ to \eqref{assume}, we obtain  
\[
\Delta_1(x):=(\gamma+\widetilde{f}(\beta, b_{\sigma})+x)^{p}-\gamma^{p}-\widetilde{f}(\beta, b_{\sigma})^{p}-x^{p} \in \mathfrak{p}[x]. 
\]
We have 
$\sigma(\beta)\widetilde{R}(\sigma(\beta))=(\beta+b_{\sigma}) \widetilde{R}(\beta+b_{\sigma})$ by substituting \eqref{abcr}  and using \eqref{a12}. 
By multiplying the first congruence in \eqref{ue} by $b_{\sigma}\beta^{-1}$, we obtain $b_{\sigma}\widetilde{R}(\beta+b_{\sigma}) 
\equiv 
b_{\sigma} \widetilde{R}(\beta)
+b_{\sigma} \widetilde{R}(b_{\sigma}) \mod 0+$.
Hence, we compute 
\begin{align*}
\sigma(\gamma)^{p}-\sigma(\gamma)&=\sigma(\beta)\widetilde{R}(\sigma(\beta))=(\beta+b_{\sigma}) \widetilde{R}(\beta+b_{\sigma}) \\
& \equiv 
\beta \widetilde{R}(\beta)+b_{\sigma}\widetilde{R}(\beta)+\beta \widetilde{R}(b_{\sigma})+b_{\sigma}\widetilde{R}(b_{\sigma}) \\
& \equiv \gamma^{p}-\gamma+\widetilde{f}(\beta, b_{\sigma})^{p}-\widetilde{f}(\beta,b_{\sigma})+b_{\sigma}\widetilde{R}(b_{\sigma}) \\
& \equiv \sigma(\gamma)^{p}-\sigma(\gamma)
-(c_{\sigma}^{p}-c_{\sigma}+\Delta_1(c_{\sigma}))
+b_{\sigma}\widetilde{R}(b_{\sigma})  
\mod 0+, 
\end{align*}
where we have used \eqref{abcr} for the last congruence. 
Hence we obtain $c_{\sigma}^p-c_{\sigma}+\Delta_1(c_{\sigma}) \equiv b_{\sigma}\widetilde{R}(b_{\sigma})
\mod 0+$. By $b_{\sigma} \in \mathcal{O}$, 
we have $c_{\sigma} \in \mathcal{O}$
and $\bar{c}_{\sigma}^p-\bar{c}_{\sigma}=\bar{b}_{\sigma}R(\bar{b}_{\sigma})$. 
\end{proof}
Assume that
\begin{gather}\label{ffd}
\begin{aligned}
& (x+\beta_{R,m})^{p^i}-x^{p^i}-\beta_{R,m}^{p^i} \in \mathfrak{p}[x] \quad \textrm{for $1 \leq i \leq e-1$}, \\
&\widetilde{f}(\beta_{R,m},x+y)-\widetilde{f}(\beta_{R,m},x)-\widetilde{f}(\beta_{R,m},y) \in \mathfrak{p}[x,y], \\
\end{aligned}
\end{gather}
Let 
\begin{equation}\label{tmr}
\Theta_{R,m,\varpi} \colon W_F \to Q_R \rtimes \mathbb{Z};\ \sigma \mapsto 
((\bar{a}_{\sigma}^m,\bar{b}_{\sigma},\bar{c}_{\sigma}), n_{\sigma}). 
\end{equation}
\begin{lemma}
The map $\Theta_{R,m,\varpi}$
is a homomorphism.
\end{lemma}
\begin{proof}
Let $\sigma, \sigma' \in W_F$. 
Recall that $\sigma(x) \equiv x^{q^{-n_{\sigma}}} \mod 0+$ for $x \in \mathcal{O}_{\overline{F}}$. 
Using Definition \ref{qdef}(2), 
we reduce to checking that 
\begin{equation}\label{abce}
\bar{a}_{\sigma \sigma'}
=\bar{a}_{\sigma} \bar{a}_{\sigma'}^{q^{-n_{\sigma}}}, \quad 
\bar{b}_{\sigma \sigma'}=
\bar{a}_{\sigma}^m \bar{b}_{\sigma'}^{q^{-n_{\sigma}}}+\bar{b}_{\sigma}, \quad 
\bar{c}_{\sigma \sigma'}=
\bar{c}_{\sigma}+\bar{c}_{\sigma'}^{q^{-n_{\sigma}}}
+f_R(\bar{b}_{\sigma},\bar{a}_{\sigma}^m \bar{b}_{\sigma'}^{-n_{\sigma}}). 
\end{equation}
We easily check that $a_{\sigma\sigma'}=\sigma(a_{\sigma'})
a_{\sigma}$ and $b_{\sigma \sigma'}=
a_{\sigma}^m \sigma(b_{\sigma'})+b_{\sigma}$. 
Hence the first equalities in \eqref{abce} follow. 
We compute 
\begin{align*}
c_{\sigma \sigma'}&=c_{\sigma}+\sigma(c_{\sigma'})
+\sigma(\widetilde{f}(\beta,b_{\sigma'}))
+\widetilde{f}(\beta,b_{\sigma})-
\widetilde{f}
(\beta,b_{\sigma \sigma'}) \\
 & \equiv c_{\sigma}+\sigma(c_{\sigma'})
+\sigma(\widetilde{f}(\beta,b_{\sigma'}))
-\widetilde{f}
(\beta,a_{\sigma}^m \sigma(b_{\sigma'})) \mod 0+, 
\end{align*}
where we use the second condition in \eqref{ffd} for the second congruence. 
We have 
\begin{align*}
\sigma(\widetilde{f}(\beta,b_{\sigma'}))&=
-\sum_{i=0}^{e-1} \sum_{j=0}^{e-i-1} (\widetilde{a}_j \sigma(b_{\sigma'}) \sigma(\beta)^{p^i})^{p^j}
-\sum_{i=0}^{e-1} (\sigma(\beta) \widetilde{R}(\sigma(b_{\sigma'})))^{p^i} \\
& \equiv \widetilde{f}(b_{\sigma}, a_{\sigma}^m\sigma(b_{\sigma'}))+\widetilde{f}
(\beta, a_{\sigma}^m \sigma(b_{\sigma'})) \mod 0+, 
\end{align*}
where we substitute $\sigma(\beta)=a_{\sigma}^{-m}(\beta+b_{\sigma})$, \eqref{ffd} and \eqref{a12}
for the second congruence. 
The last equality in \eqref{abce} follows 
from $\overline{\widetilde{f}(b_{\sigma}, a_{\sigma}^m\sigma(b_{\sigma'}))}=f_R
(\bar{b}_{\sigma},\bar{a}_{\sigma}^m \bar{b}_{\sigma'}^{q^{-n_{\sigma}}})$, since 
$\widetilde{f}(x,y)$ is a lift of 
$f_R(x,y)$ to $\mathcal{O}_F[x,y]$. 
\end{proof}
\begin{lemma}\label{necr}
If $v(p)$ is large enough, 
the conditions \eqref{assume} and \eqref{ffd} are satisfied. 
\end{lemma}
\begin{proof}
There exists 
$s \in \mathbb{Z}_{\geq 1}$ such that 
the coefficients of 
all polynomials  
in  \eqref{assume} and \eqref{ffd} have the form: 
$p \cdot \beta_{R,m}^s \cdot a$ with $a \in \mathcal{O}_{\overline{F}}$ by Remark \ref{rb}. 
Since the valuation of $v(\beta_{R,m})$ is 
independent of $F$, 
the claim follows.   
\end{proof}
In the sequel, we assume that the conditions \eqref{assume} and \eqref{ffd} are satisfied. 
Let $F^{\rm ur}$ denote the maximal 
unramified extension of $F$ in $\overline{F}$. 
\begin{lemma}
The extension $F^{\rm ur}(\alpha_R^m,\beta_{R,m},\gamma_{R,m})/F$ is Galois. 
\end{lemma}
\begin{proof}
Let $L_0:=F^{\rm ur}(\alpha^m,\beta,\gamma)$ and $L:=\widehat{L}_0$ be the completion of $L_0$. 
Let $\sigma \in G_F$. 
We note that $a_{\sigma} \in 
\mu_{d_R}(\overline{F}) \subset F^{\rm ur}$ by $p \nmid d_R$. 
Hence $\sigma(\alpha^m)=a_{\sigma}^m \alpha^m \in L_0$. 
We show 
$\sigma(\beta), \sigma(\gamma) \in L_0$. 
It suffices to prove
\[
b_{\sigma}, c_{\sigma} \in L_0, 
\]
since 
\[
\sigma(\beta)=\frac{b_{\sigma}+\beta}{a_{\sigma}^m}, \quad 
\sigma(\gamma)=\gamma+c_{\sigma}+
\widetilde{f}(\beta,b_{\sigma})
\]
by \eqref{abcr}. 
As in the proof of Lemma \ref{abbc}, we have 
$(\widetilde{E}_R+\Delta)(b_{\sigma})=0$,  $E(x):=(\widetilde{E}_R+\Delta)(x) \in \mathcal{O}_{L_0}[x]$ and 
$\deg E(x)=p^{2e}$. 
The equation $E(x) \equiv 0 \mod 0+$
has $p^{2e}$ different roots. 
Hence by Hensel's lemma, 
$E(x)=0$ has $p^{2e}$ different roots 
in $\mathcal{O}_L$. 
Hence we have $b_{\sigma} \in L \cap \overline{F}
=L_0$. 
As in the proof of Lemma \ref{abbc}, 
we have 
\[
f(c_{\sigma}):=c_{\sigma}^p-c_{\sigma}+
\Delta_1(c_{\sigma})-y=0\ \textrm{with $y \in \mathcal{O}_{L_0}$},  
\]
where $f(x) \in \mathcal{O}_{L_0}[x]$ with $\deg f(x)=p$.  We have $y \equiv b_{\sigma} \widetilde{R}(b_{\sigma})
\mod 0+$.  
The equation $f(x) \equiv x^p-x- y \equiv x^p-x-b_{\sigma} \widetilde{R}(b_{\sigma}) \equiv 0 \mod 0+$ 
has $p$ different roots. 
By Hensel's lemma, $f(x)=0$ has $p$ different roots in 
$\mathcal{O}_L$. Hence we have $c_{\sigma} \in L \cap \overline{F}=L_0$.
\end{proof}
\begin{definition}
Let
\[
d_{R,m}:=\frac{d_R}{\mathrm{gcd}(d_R,m)}, \quad 
Q_{R,m}:=\{(\alpha,\beta,\gamma) \in Q_R \mid \alpha \in 
\mu_{d_{R,m}}\}. 
\] 
\end{definition}
\begin{lemma}\label{3.8}
We have the isomorphism 
\begin{equation}\label{weil}
W(F^{\rm ur}(\alpha_R^m,\beta_{R,m},\gamma_{R,m})/F) \xrightarrow{\sim}
Q_{R,m} \rtimes \mathbb{Z};\ \sigma \mapsto ((\bar{a}_{\sigma}^m,\bar{b}_{\sigma},\bar{c}_{\sigma}),n_{\sigma}). 
\end{equation}
\end{lemma}
\begin{proof}
We use the notation in the proof of Lemma \ref{abbc}. 
Let 
\[
I:=\Gal(F^{\rm ur}(\alpha^m,\beta,\gamma)/F^{\rm ur})
\]
and $\Theta \colon I \to Q_{R, m}$
be the restriction of \eqref{weil}. 
First, we show that $\Theta$ is injective.  
Assume $\Theta(\sigma)=1$ for $\sigma \in I$. 
We will show $\sigma=1$. 
By the assumption, $\bar{a}_{\sigma}^m=1$, 
$\bar{b}_{\sigma}=0$ and $\bar{c}_{\sigma}=0$. 
We have a natural isomorphism 
$\mu_r(\overline{F}) \xrightarrow{\sim} 
\mu_r$ for an integer $r$ prime to $p$. 
By $\bar{a}_{\sigma}^m=1$ and 
$a_{\sigma} \in \mu_{d_R}(\overline{F})$
in \eqref{abcr}, we have $a_{\sigma}^m=1$ and $\sigma(\alpha^m)=\alpha^m$. 
We recall the equality $\widetilde{E}_R(b_{\sigma})+\Delta(b_{\sigma})=0$ in the proof of Lemma \ref{abbc}, where $\Delta(x) \in \mathfrak{p}[x]$ has no constant 
coefficient. 
We have $v(b_{\sigma})>0$ by $\bar{b}_{\sigma}=0$. 
Assume that $b_{\sigma} \neq 0$. 
Then $v(b_{\sigma})=v(\widetilde{E}_R(b_{\sigma})+\Delta(b_{\sigma}))$ by $E'_R(0) \neq 0$ and $v(b_{\sigma})>0$. This implies that
$v(b_{\sigma})=\infty$, which is a contradiction. 
Hence $b_{\sigma}=0$ and $\sigma(\beta)=\beta$. 
By the last condition in \eqref{assume}
with $y=0$, 
\[
\Lambda(x):=(\gamma+x)^p-\gamma^p-x^p \in \mathfrak{p}[x]. 
\]
We have $\sigma(\gamma)^p-\sigma(\gamma)=\gamma^p-\gamma$ by Definition \ref{3d}. 
Hence $(\gamma+c_{\sigma})^p-\gamma^p=c_{\sigma}$ and 
$c_{\sigma}^p+\Lambda(c_{\sigma})=c_{\sigma}$. 
Since $\Lambda(x) \in \mathfrak{p}[x]$ has no constant coefficient, if $0<v(c_{\sigma})<\infty$, 
we have $v(c_{\sigma}^p+\Lambda(c_{\sigma}))>v(c_{\sigma})$, which 
can not occur. 
Hence $c_{\sigma}=0$ and $\sigma(\gamma)=\gamma$. 
We obtain $\sigma = 1$. 
Hence $\Theta$ is injective. 
We easily check that $F^{\rm ur}(\alpha^m,\beta,\gamma)/F^{\rm ur}$ is a 
totally ramified extension of degree $d_{R,m} p^{2e+1}$.
Hence $\Theta$ is an isomorphism.  
By the snake lemma, \eqref{weil} is an isomorphism. 
\end{proof}
\subsection{Galois representations associated to
additive polynomials}
\label{3.2}
\subsubsection{Construction of Galois representation}
We assume that \eqref{assume} and \eqref{ffd} are satisfied. If the characteristic of 
$F$ is positive, these are unconditional. If the characteristic of $F$ is 
zero, these conditions are satisfied if the absolute ramification index of $F$ is large enough as in Lemma \ref{necr}. 
\begin{definition}
Let $\psi \in \mathbb{F}_{p}^{\vee} \setminus 
\{1\}$. We define $\tau_{\psi,R,m,\varpi}$ to be the $W_F$-representation which is the inflation of the irreducible $Q_R \rtimes \mathbb{Z}$
-representation $H_{\rm c}^1
(C_{R,\mathbb{F}},\overline{\mathbb{Q}}_{\ell})[\psi]$ by $\Theta_{R,m,\varpi}$ in \eqref{tmr}.
For simplicity, we write  
$\tau_{\psi,R,m}$ for $\tau_{\psi,R,m,\varpi}$. 
\end{definition} 
For a non-archimedean local field $K$, 
let $I_K$ denote the inertia subgroup of $K$. 
Then $\Ker \tau_{\psi,R,m}$ contains the open compact subgroup 
$I_{F(\alpha_R^m,\beta_{R,m},\gamma_{R,m})}$. 
Hence the representation $\tau_{\psi,R,m}$ is a smooth irreducible representation of $W_F$ by Lemma \ref{fa}(1).

Let $G_F:=\Gal(\overline{F}/F)$. 
We consider a general setting in the following 
lemma. 
\begin{lemma}\label{bel}
Let $\widetilde{\tau}$ be a continuous 
representation of $G_F$ over $\overline{\mathbb{Q}}_{\ell}$ such that 
there exists an unramified continuous 
character $\phi$ of 
$G_F$ such that $(\widetilde{\tau} \otimes \phi)(G_F)$ is finite. 
Assume that $\tau:=\widetilde{\tau}|_{W_F}$ is irreducible. 
Then $\widetilde{\tau} \otimes \phi$ is primitive if and only if $\tau$ is primitive.
\end{lemma} 
\begin{proof}
Let 
$\widetilde{\tau}':=\widetilde{\tau} \otimes \phi$ and 
$\tau':=\tau \otimes \phi|_{W_F}$. 
The subgroup $\Ker \widetilde{\tau}'$ is open by $|G_F/\Ker \widetilde{\tau}'|<\infty$. 
Hence $\Ker \tau' \subset W_F$ is
open. Therefore   
$\tau'$ is smooth. 
Hence so is $\tau$. 
Since $\tau$ is irreducible and smooth, we have $\dim \tau<\infty$. 
We will show that 
$\widetilde{\tau}'$ is imprimitive if and only if 
$\tau$ is imprimitive. 

First, assume an isomorphism 
$\widetilde{\tau}' \simeq \Ind_{H}^{G_F} \eta'$
with a proper subgroup $H$.
We can check $\Ker \widetilde{\tau}' \subset H$. 
Hence $H$ is open.  
Hence we can write  $H=G_{F'}$ with a finite extension $F'/F$. 
Hence we obtain an isomorphism 
$\tau \simeq \Ind_{W_{F'}}^{W_F} (\eta'|_{W_{F'}} \otimes \phi^{-1}|_{W_{F'}})$.

To the contrary, assume $\tau \simeq \Ind_H^{W_F} \sigma$. 
In the same manner as above with replacing 
$G_F$ by $W_F$, the subgroup $H$ is 
an open subgroup of $W_F$ of finite index
by $\dim \tau<\infty$. 
Hence we can write $H=W_{F'}$ with a finite extension $F'/F$. Let $\sigma':=\sigma \otimes \phi|_{W_{F'}}$. 
We have $\tau' \simeq 
\Ind_{W_{F'}}^{W_F} \sigma'$.  
Frobenius reciprocity implies that 
$\sigma'(W_{F'})
\subset \tau'(W_F)$. 
By the assumption, the image $\sigma'(W_{F'})$ is finite. Hence the smooth $W_{F'}$-representation $\sigma'$ extends to a smooth representation of $G_{F'}$,
for which we write  $\widetilde{\sigma}$ (\cite[Proposition 28.6]{BH}).  
The restriction of $\Ind_{G_{F'}}^{G_F} \widetilde{\sigma}$ to $W_F$ is 
isomorphic to 
$\Ind_{W_{F'}}^{W_F} \sigma' \simeq \tau'$. 
Both of  $\Ind_{G_{F'}}^{G_F} \widetilde{\sigma}$
and $\widetilde{\tau}'$ are smooth 
irreducible $G_F$-representations whose restrictions to $W_F$ are 
isomorphic to $\tau'$. 
Hence we obtain an isomorphism $\widetilde{\tau}' \simeq \Ind_{G_{F'}}^{G_F} \widetilde{\sigma}$ as $G_F$-representations by \cite[Lemma 28.6.2(2)]{BH}. 
\end{proof}

\begin{lemma}\label{rem1}
The eigenvalues of $\mathrm{Fr}_q^\ast$ on $H_{\rm c}^1(C_{R,\mathbb{F}},\overline{\mathbb{Q}}_{\ell})[\psi]$
have the forms $\zeta \sqrt{q}$ with roots of unity 
$\zeta$. The automorphism $\mathrm{Fr}_q^\ast$
is semi-simple. 
\end{lemma}
\begin{proof}
The claims follow from Proposition \ref{super}.   
\end{proof}
The cohomology group $H_{\rm c}^1
(C_{R,\mathbb{F}},\overline{\mathbb{Q}}_{\ell})[\psi]$ is regarded as a representation of 
$Q_{R,m} \rtimes \widehat{\mathbb{Z}}$. 
By inflating this by a natural map 
$G_F \to Q_{R,m} \rtimes \widehat{\mathbb{Z}}$ extending $\Theta_{R,m,\varpi}$, we obtain a continuous representation of 
$G_F$. We denote this representation by $\widetilde{\tau}_{\psi,R,m}$. 
Let 
$\phi \colon G_F \to \overline{\mathbb{Q}}_{\ell}^{\times}$ be the unramified 
character sending a geometric Frobenius 
to $\sqrt{q}^{-1}$.  
The image of 
$G_F$ by the twist $\widetilde{\tau}':=\widetilde{\tau}_{\psi,R,m} \otimes \phi$ is finite by Lemma \ref{rem1}. 
By fixing an isomorphism 
$\overline{\mathbb{Q}}_{\ell} \simeq 
\mathbb{C}$, we obtain a continuous 
representation $\widetilde{\tau}'_{\mathbb{C}}$ of $G_F$ over $\mathbb{C}$ by $\widetilde{\tau}'$.  
Then 
$\widetilde{\tau}'_{\mathbb{C}}$ is primitive if and only if 
$\widetilde{\tau}_{\psi,R,m}$ is primitive. 
\begin{corollary}\label{belc}
The $W_F$-representation  
$\tau_{\psi,R,m}$ is primitive if and only if 
the continuous $G_F$-representation 
$\widetilde{\tau}'_{\mathbb{C}}$ is primitive. 
\end{corollary}
\begin{proof}
Clearly $\widetilde{\tau}'_{\mathbb{C}}$ is primitive if and only if 
$\widetilde{\tau}'$ is primitive.
We obtain the claim  by applying Lemma \ref{bel} with 
$\widetilde{\tau}=\widetilde{\tau}_{\psi,R,m}$ and  
$\tau=\tau_{\psi,R,m}$. 
\end{proof}
\subsubsection{Swan conductor exponent}\label{322}
In the sequel, we compute the Swan conductor exponent $\mathrm{Sw}(\tau_{\psi,R,m})$. 
 
We simply write  $\alpha,\beta,\gamma$
for $\alpha_R,\beta_{R,m},\gamma_{R,m}$ in
Definition \ref{3d}, respectively. 
We consider the unramified field extension $F_r/F$  
of degree $r$ such that 
$N:=F_r(\alpha,\beta,\gamma)$ is Galois over $F$. 
Let $T:=F_r(\alpha)$ and $M:=T(\beta)$. 
Then we have 
\[
F \subset F_r \subset T \subset M
\subset N. 
\]

Let $L/K$ be a Galois extension of non-archimedean local fields with Galois group 
$G$. 
Let $\left\{G^i\right\}_{i \geq -1}$ denote the upper numbering ramification 
groups of $G$ in \cite[IV\S3]{Se}. 
Let $\psi_{L/K}$ denote the Herbrand function of $L/K$. 
\begin{lemma}\label{her}
Let $G:=\Gal(N/F)$. 
Then we have  
\[
\psi_{N/F}(t)=
\begin{cases}
t & \textrm{if $t \leq 0$}, \\
d_R t & \textrm{if $0<t \leq \frac{m}{d_R}$}, \\
p^{2e} d_R t-(p^{2e}-1)m & \textrm{if 
$\frac{m}{d_R}<t \leq \frac{p^e+1}{p^e} \frac{m}{d_R}$}, \\
p^{2e+1} d_R t-(p^e+1)(p^{e+1}-1)m & \textrm{otherwise} 
\end{cases}
\]
and 
\[
G^i=\begin{cases}
G & \textrm{if $i =-1$}, \\
\Gal(N/F_r) & \textrm{if $-1<i \leq 0$}, \\
\Gal(N/T) & \textrm{if $0<i \leq \frac{m}{d_R}$}, \\
\Gal(N/M) & \textrm{if $\frac{m}{d_R} <i \leq \frac{p^e+1}{p^e }\frac{ m}{d_R}$}, \\
\{1\} & \textrm{otherwise}. 
\end{cases}
\]
\end{lemma}
\begin{proof}
We easily have  
\[
\psi_{T/F}(t)=\begin{cases}
t & \textrm{if $t \leq 0$}, \\
d_R t & \textrm{otherwise}. 
\end{cases}
\]
For a finite Galois extension $L/K$, 
let $\{\Gal(L/K)_u\}_{u \geq -1}$ be the lower numbering ramification subgroups. 
Let $1 \neq \sigma \in \Gal(M/T)$. 
Let $b_{\sigma}=\sigma(\beta)-\beta$
as before. 
We have $\widetilde{E}_R(\beta+b_{\sigma})
=\widetilde{E}_R(\beta)$ by the proof 
of Lemma \ref{abbc}. 
If $v(b_{\sigma})>0$, we obtain 
$b_{\sigma}=0$ by the same argument in the proof of 
Lemma \ref{3.8}. This implies that $\sigma=1$. 
By \eqref{assume}, 
we have $v(b_{\sigma})=0$. 
By $v_M(\beta)=-m$, we have 
$v_M(\sigma(\varpi_M)-\varpi_M)=m+1$.  
Hence 
\[
\Gal(M/T)_u=\begin{cases}
\Gal(M/T) &\textrm{if $u \leq m$}, \\
\{1\} & \textrm{otherwise}, 
\end{cases} \quad 
\psi_{M/T}(t)=\begin{cases}
t & \textrm{if $t \leq m$}, \\
p^{2e}t-(p^{2e}-1)m & \textrm{otherwise}. 
\end{cases}
\]
Let $1 \neq \sigma \in \mathrm{Gal}(N/M)$.  
If $v_N(\sigma(\gamma)-\gamma)>0$, 
we obtain $\sigma(\gamma)=\gamma$
in the same way as the proof of Lemma \ref{3.8}. 
This implies that $\sigma=1$.  
Hence $v_N(\sigma(\gamma)-\gamma)=0$. Let $\varpi_N$ be a uniformizer of $N$. 
By $v_N(\gamma^{-1})=(p^e+1)m$, we have 
$v_N(\sigma(\varpi_N)-\varpi_N)=(p^e+1)m+1$. 
Thus 
\begin{align*}
\Gal(N/M)_u&=\begin{cases}
\Gal(N/M) &\textrm{if $u \leq (p^e+1)m$}, \\
\{1\} & \textrm{otherwise}, 
\end{cases} \\ 
\psi_{N/M}(t)&=
\begin{cases}
t & \textrm{if $t \leq (p^e+1)m$}, \\
p t-(p-1)(p^e+1)m & \textrm{otherwise}. 
\end{cases}
\end{align*}
Hence the former claim follows
from $\psi_{N/F}=\psi_{N/M} \circ \psi_{M/T}
\circ \psi_{T/F}$. 

We can check 
\begin{align*}
G_u
=\begin{cases}
G & \textrm{if $u =-1$}, \\
\Gal(N/F_r) & \textrm{if $-1<u \leq 0$}, \\
\Gal(N/T) & \textrm{if $0<u \leq m$}, \\
\mathrm{Gal}(N/M) & \textrm{if $m<u \leq (p^e+1)m$}, \\
\{1\} & \textrm{otherwise} 
\end{cases}
\end{align*}
 by using the former claim and 
\cite[Propositions 12(c), 13(c) and 15 in IV\S3]{Se}. 
Hence the latter claim follows from
$G^i=G_{\psi_{N/F}(i)}$. 
\end{proof}
\begin{corollary}\label{indd}
We have $\mathrm{Sw}(\tau_{\psi,R,m})=
m(p^e+1)/d_R$. 
\end{corollary}
\begin{proof}
Recall that the twist $\tau_{\psi,R,m} \otimes \phi$
factors through a finite group 
$Q_R \rtimes 
(\mathbb{Z}/r \mathbb{Z}) \simeq 
\Gal(F_r(\alpha,\beta,\gamma)/F)$ 
with a certain integer $r$. 
Since $\phi$ is unramified, we have 
$\mathrm{Sw}(\tau_{\psi,R,m})=\mathrm{Sw}(\tau_{\psi,R,m} \otimes \phi)$. 
We have 
$\mathrm{Sw}(\tau_{\psi,R,m} \otimes \phi)
=m(p^e+1)/d_R$ by 
 Lemma \ref{her} and \cite[Th\'eor\`eme 7.7]{He} (\cite[Exercise 2 in \S2VI]{Se}). 
\end{proof}
\subsection{Symplectic module associated to 
Galois representation}\label{3.3}
Let $\rho \colon W_F \to \mathrm{PGL}_{p^e}(\overline{\mathbb{Q}}_{\ell})$ denote the composite of $\tau_{\psi,R,m} \colon W_F \to \mathrm{GL}_{p^e}(\overline{\mathbb{Q}}_{\ell})$ with the natural map 
$\mathrm{GL}_{p^e}(\overline{\mathbb{Q}}_{\ell}) \to \mathrm{PGL}_{p^e}(\overline{\mathbb{Q}}_{\ell})$. 

Let $\rho'$ be the projective representation 
associated to $\widetilde{\tau}'=\widetilde{\tau}_{\psi,R,m} \otimes \phi$. 
\begin{lemma}
We have $\rho(W_F)=\rho'(G_F)$, which is finite. 
\end{lemma} 
\begin{proof}
Since $\widetilde{\tau}'$ is a  smooth irreducible $G_F$-representation, 
we have 
$\widetilde{\tau}'(G_F)
=(\tau_{\psi,R,m} \otimes \phi)(W_F)$ (\cite[the proof of Lemma 2 in 28.6]{BH}). This implies the claim.  
\end{proof}
Let $F_{\rho}$ denote the kernel field 
of $\rho$ and $T_{\rho}$ the maximal tamely ramified extension of $F$ in $F_{\rho}$. Let 
\[
H:=\mathrm{Gal}(F_{\rho}/T_{\rho}) \subset G:=\mathrm{Gal}(F_{\rho}/F).  
\] 
The homomorphism $\rho$ induces
an injection  
$\bar{\rho} \colon G \to \mathrm{PGL}_{p^e}(\overline{\mathbb{Q}}_{\ell})$. 
Let $V_R$ be as in Lemma \ref{ab}. 

Let $\tau$ denote the composite   
\[
Q_{R,m} \rtimes \mathbb{Z} \to \mathrm{Aut}_{\overline{\mathbb{Q}}_{\ell}}(H_{\rm c}^1(C_{R,\mathbb{F}},\overline{\mathbb{Q}}_{\ell})[\psi]) \to \mathrm{Aut}_{\overline{\mathbb{Q}}_{\ell}}(H_{\rm c}^1(C_{R,\mathbb{F}},\overline{\mathbb{Q}}_{\ell})[\psi])/\overline{\mathbb{Q}}_{\ell}^{\times}. 
\]
\begin{lemma}\label{asso}
We have an isomorphism
$\bar{\rho}(H)\simeq V_R$.
\end{lemma}
\begin{proof}
Let 
 $L:=F^{\rm ur}(\alpha_R^m,\beta_{R,m},\gamma_{R,m})$ and $K:=F^{\rm ur}(\alpha_R^m)$. 
By \eqref{weil}, we have  
\[
W(L/F) \simeq Q_{R,m} \rtimes \mathbb{Z}, \quad  
W(L/K) \simeq H_R. 
\]
The subfield $K$ is the maximal tamely ramified 
extension of $F$ in $L$. 
We have $F_{\rho} \subset L$ and 
$T_{\rho}=F_{\rho} \cap K$.  
We have isomorphisms $G=W(F_{\rho}/F) \simeq W_F/\Ker \rho \simeq (Q_{R,m} \rtimes \mathbb{Z})/\Ker \tau$ 
and  
$H=\mathrm{Gal}(F_{\rho}/T_{\rho}) \simeq H_R/(H_R \cap \Ker \tau)$. 
By  Lemma \ref{inj}, 
we have $H_R \cap \Ker \tau=Z(H_R)$. 
Hence the claim follows from the isomorphism 
$V_R \xrightarrow{\sim} H_R/Z(H_R);\ \beta \mapsto (1,\beta,0)$. 
 \end{proof}
Let 
 \[
 \mathscr{H}_0:=G/H=\mathrm{Gal}(T_{\rho}/F). 
 \]
Let $\sigma \in \mathscr{H}_0$. 
We take a lifting $\widetilde{\sigma} \in G \twoheadrightarrow \mathscr{H}_0$ of $\sigma$.
Let $\mathscr{H}_0$ act on $H$ by  
 $\sigma \cdot \sigma':=\widetilde{\sigma} \sigma'\widetilde{\sigma}^{-1}$ for $\sigma' \in H$. 
This is well-defined because
$H$ is abelian by Lemma \ref{asso}. 
We regard $H \simeq V_R$ as an 
$\mathbb{F}_p[\mathscr{H}_0]$-module. 

By Lemma \ref{rem1}, we can take a
positive integer $r$ such that $r \mathbb{Z} \subset \Ker \tau$ and $x^{q^{r}}=x$ for $x \in \mu_{d_{R,m}}$. 
Let $\mathbb{Z}/r \mathbb{Z}$
act on $\mu_{d_{R,m}}$ by $1 \cdot x=x^{q^{-1}}$. 
We take a generator $\alpha \in \mu_{d_{R,m}}$. 
Let 
\begin{equation}\label{hr}
\mathscr{H}:=\mu_{d_{R,m}} \rtimes (\mathbb{Z}/r \mathbb{Z})\xrightarrow{\sim} \left\langle \sigma,\tau \mid \sigma^r=1, \   
\tau^{d_{R,m}}=1, \  
\sigma \tau \sigma^{-1}=\tau^q \right\rangle,  
\end{equation}
where the isomorphism is given by 
$(\alpha,0) \mapsto \tau$ and $(1,-1) \mapsto 
\sigma$. The groups $\mathscr{H}_0$
and $\mathscr{H}$ are supersolvable. 
We consider the commutative diagram 
\[
\xymatrix{
Q_{R,m} \rtimes \mathbb{Z} \ar[r]\ar[d] & (Q_{R,m} \rtimes \mathbb{Z})/\Ker \tau\simeq G
\ar[d] \\
\mathscr{H} \simeq 
(Q_{R,m} \rtimes \mathbb{Z})/(H_R \rtimes r \mathbb{Z})
 \ar[r]^-{\varphi} & (Q_{R,m} \rtimes \mathbb{Z})/(\Ker \tau \cdot H_R)
 \simeq \mathscr{H}_0,}  
\]
where every map is canonical and surjective. 
\begin{lemma}\label{sym0}
The elements $\varphi(\alpha,0)$
and $\varphi(1,-1)$ in $\mathscr{H}_0$ 
act on $H \simeq V_R$ by $x \mapsto \alpha x$ and $x \mapsto x^q$ for $x \in V_R$, respectively. 
\end{lemma}
\begin{proof}
These are directly checked. 
\end{proof}
We can regard $V_R$ as an $\mathbb{F}_p[\mathscr{H}]$-module via $\varphi$. 
 Let $\omega_R$ be as in Lemma \ref{ab}(2). 
 \begin{lemma}\label{sym}
We have $\omega_R(hx,hx')=\omega_R(x,x')$
for $h \in \mathscr{H}$.  
 \end{lemma}
 \begin{proof}
 The claim for $h=\alpha$ 
 follows from 
 \eqref{a-1}. 
 For $h=(1,-1)$, the claim follows 
 from
 $\omega_R(x^q,x'^q)=(f_R(x,x')-f_R(x',x))^q
 =f_R(x,x')-f_R(x',x)=\omega_R(x,x')$. 
 \end{proof}
 \begin{definition}
 Let $G$ be a finite group. 
 Let $V$ be an $\mathbb{F}_p[G]$-module with 
 $\dim_{\mathbb{F}_p} V <\infty$. 
 Let $\omega \colon V \times V \to \mathbb{F}_p$
 be a symplectic form. 
 We say that the pair 
 $(V,\omega)$ is \textit{symplectic} if $\omega$ is non-degenerate and satisfies  
 $\omega(gv,gv')=\omega(v,v')$ for $g \in G$ and $v,v' \in V$.  
 \end{definition}
\begin{lemma}
The $\mathbb{F}_p[\mathscr{H}]$-module 
$(V_R,\omega_R)$ is symplectic. 
\end{lemma}
\begin{proof}
The claim follows from 
Lemma \ref{ab}(2) and Lemma \ref{sym}. 
\end{proof}
\begin{definition}
The $\mathbb{F}_p[\mathscr{H}_0]$-module 
$(V_R,\omega_R)$ is called a symplectic module associated to 
$\tau_{\psi,R,m}$. 
\end{definition}
\begin{definition}\label{eled}
Let $\sigma \colon \mathbb{F} \to \mathbb{F};\ 
x \mapsto x^q$. 
For $f(x)=\sum_{i=0}^n a_i x^i \in \mathbb{F}[x]$, we set 
$f^{\sigma}(x):=\sum_{i=0}^n \sigma(a_i) x^i$. 
\end{definition}
Let $k$ be a field. 
We say that a polynomial $f(x) \in k[x]$ is 
reduced if the ring $k[x]/(f(x))$ is reduced. 
 \begin{lemma}\label{ele}
 Let $E(x) \in \mathscr{A}_{\mathbb{F}}$ be 
 a reduced polynomial. 
 Let $V:=\{\beta \in \mathbb{F} \mid E(\beta)=0\}$. 
\begin{itemize}
\item[{\rm (1)}] 
Assume that $E(x)$ is monic and 
$V$ is stable under $\sigma$.  
Then we have $E(x) \in \mathbb{F}_q[x]$.  
\item[{\rm (2)}]  
Let $r$ be a positive integer. 
 Assume that $V$ is stable under $\mu_r$-multiplication. Then we have $E(\alpha x)=\alpha E(x)$ for $\alpha \in \mu_r$.     
 \end{itemize}
 \end{lemma}
 \begin{proof}
Recall that 
$f(x) \in \mathscr{A}_{\mathbb{F}}$ is reduced if and only if $f'(0) \neq 0$. 

We show (1). 
By the assumption, we have $E^{\sigma}(\beta)=(E(\beta^{q^{-1}}))^q=0$ for any $\beta \in V$. Since $E(x)$ is separable,
there exists $\alpha \in \mathbb{F}^{\times}$
such that $E^{\sigma}(x)=\alpha E(x)$. 
Since $E(x)$ is monic, we have $\alpha=1$. Hence we have the claim.

 We show (2). 
 Let $\alpha \in \mu_r$. 
 By the assumption, 
 $E(\alpha \beta)=0$ for any $\beta \in V$. Since $E(x)$ is separable, we have $E(\alpha x)=c E(x)$
 with a constant $c \in \mathbb{F}^{\times}$. By considering the derivatives of 
$E(\alpha x)$, $cE(x)$ and substituting $x=0$, 
 we have $\alpha=c$ by $E'(0) \neq 0$. Hence the claim follows. 
  \end{proof}
  \begin{definition}\label{24}
  Let $f(x) \in \mathscr{A}_q$. 
  \begin{itemize}
\item[{\rm (1)}] A decomposition $f(x)=f_1(f_2(x))$ with $f_i(x) \in \mathscr{A}_q$ is said to be \textit{non-trivial} if 
 $\deg f_i>1$ for $i \in \{1,2\}$. 
\item[{\rm (2)}] We say that 
$f(x) \in \mathscr{A}_q$ is \textit{prime} if 
it does not admit a non-trivial decomposition $f(x)=f_1(f_2(x))$ with 
$f_i(x) \in \mathscr{A}_q$. 
\end{itemize}
\end{definition}
\begin{definition}
Let $(V,\omega)$ be a symplectic $\mathbb{F}_p[\mathscr{H}]$-module. 
Then $(V,\omega)$ is said to be \textit{completely anisotropic} if 
 $V$ does not admit a non-zero totally isotropic $\mathbb{F}_p[\mathscr{H}]$-submodule. 
\end{definition}
For an $\mathbb{F}_p$-subspace $W \subset V$, let 
 $W^{\perp}:=\{v \in V \mid  \textrm{$\omega(v,w)=0$ for all $w \in W$}\}$. 
 \begin{proposition}\label{symp}
 The symplectic $\mathbb{F}_p[\mathscr{H}]$-module 
 $(V_R,\omega_R)$ is completely anisotropic if and only if there does not exist a non-trivial decomposition 
 $E_R(x)=f_1(f_2(x))$ with $f_i(x) \in 
 \mathscr{A}_q$ such that $f_2(\alpha x)=\alpha f_2(x)$ for $\alpha \in \mu_{d_{R,m}}$ and $V_{f_2}:=\{\beta \in \mathbb{F} \mid 
 f_2(\beta)=0\}$ satisfies $V_{f_2} \subset V_{f_2}^{\perp}$. 
  \end{proposition}
 \begin{proof}
 Assume that there exists such a decomposition
 $E_R(x)=f_1(f_2(x))$. Since the decomposition is non-trivial, we have $V_{f_2} \neq \{0\}$. 
Hence $V_{f_2}$ is a non-zero totally isotropic $\mathbb{F}_p[\mathscr{H}]$-submodule of $V_R$.  Hence $V_R$ is not completely anisotropic. 

Assume that $V_R$ is not completely anisotropic. We take a non-zero totally isotropic $\mathbb{F}_p[\mathscr{H}]$-submodule $V' \subset V_R$. By \cite[4 in Chapter 1]{Ore}, there exists a monic reduced polynomial 
$f(x) \in \mathscr{A}_{\mathbb{F}}$
such that $V'=\{\beta \in \mathbb{F}\mid f(\beta)=0\}$. Since $V'$ is stable by $\sigma$, we have $f(x) \in \mathbb{F}_q[x]$ by Lemma \ref{ele}(1). Since $V'$ is stable by $\tau$, we have 
$f(\alpha x)=\alpha f(x)$ for $\alpha \in 
\mu_{d_{R,m}}$ by Lemma \ref{sym0} and Lemma \ref{ele}(2). 
There exist  
$f_1(x),r(x) \in \mathscr{A}_q$ such that 
$E_R(x)=f_1(f(x))+r(x)$ and $\deg r(x)<\deg f(x)$
by \cite[Theorem 1]{Ore}. For any root $\beta \in V'$
of $f(x)$, we have $r(\beta)=0$ by $E_R(\beta)=0$. Since $f(x)$ is separable, $r(x)$ is divisible by
$f(x)$. Hence $r(x) \equiv 0$ by $\deg r(x)<\deg f(x)$.  
By definition, 
we have $V' \subset V'^{\perp}$. Hence the 
converse is shown. 
 \end{proof} 
 \begin{corollary}\label{mc}
 \begin{itemize}
 \item[{\rm (1)}] 
 The $W_F$-representation $\tau_{\psi,R,m}$ is primitive if and only if the symplectic $\mathbb{F}_p[\mathscr{H}]$-module 
 $(V_R,\omega_R)$ is completely anisotropic. 
\item[{\rm (2)}] The $W_F$-representation $\tau_{\psi,R,m}$ is primitive if and only if there does not exist a non-trivial decomposition 
 $E_R(x)=f_1(f_2(x))$ with $f_i(x) \in 
 \mathscr{A}_q$ such that 
 $f_2(\alpha x)=\alpha f_2(x)$ for 
 $\alpha \in \mu_{d_{R,m}}$ and $V_{f_2}:=\{\beta \in \mathbb{F} \mid 
 f_2(\beta)=0\}$ satisfies $V_{f_2} \subset V_{f_2}^{\perp}$.
 \item[{\rm (3)}] 
 If $E_R(x) \in \mathscr{A}_q$ is prime,  
 the $W_F$-representation $\tau_{\psi,R,m}$ is primitive. 
 \item[{\rm (4)}] 
 If $R(x)=a_e x^{p^e}$ and 
 $\mathbb{F}_p(\mu_{d_{R,m}})=\mathbb{F}_{p^{2e}}$, the $\mathbb{F}_p[\mathscr{H}]$-module 
 $V_R$ is irreducible. 
 The $W_F$-representation $\tau_{\psi,R,m}$ is primitive. If $\mathrm{gcd}(p^e+1,m)=1$,
 the condition $\mathbb{F}_p(\mu_{d_{R,m}})=\mathbb{F}_{p^{2e}}$ is satisfied. 
 \end{itemize}  
 \end{corollary}
 \begin{proof}
 The claim (1) follows from Corollary \ref{belc}, 
 Lemma \ref{asso}, 
  and \cite[Theorem 4.1]{K}.

 The claim (2) follows from the claim (1) and Proposition 
 \ref{symp}. 
The claim (3) follows from (2) immediately. 

We show (4). 
We assume that there exists a non-zero 
$\mathbb{F}_p[\mathscr{H}]$-submodule 
$W \subset V_R=\{\beta \in \mathbb{F} \mid (a_e x^{p^e})^{p^e}+a_ex=0 \}$. 
We take a non-zero element
$\beta \in W$. By $\mathbb{F}_p(\mu_{d_{R,m}})=\mathbb{F}_{p^{2e}}$, 
we have $\mathbb{F}_{p^{2e}} \beta=\mathbb{F}_p(\mu_{d_{R,m}})\beta \subset W$. 
 Since $V_R$ is the set of the roots of a separable  polynomial $E_R(x)$ of degree $p^{2e}$, 
 we have $|V_R|=p^{2e}$. 
Hence $W=V_R=\mathbb{F}_{p^{2e}} \beta$. Thus the first claim follows. 
The second claim follows 
from the first one and \cite[Theorem 4.1]{K}. 
If $\mathrm{gcd}(p^e+1,m)=1$, we have 
$d_{R,m}=d_R=p^e+1$. 
Hence the third claim follows from 
$\mathbb{F}_p(\mu_{p^e+1})=
\mathbb{F}_{p^{2e}}$. 
 \end{proof}
 \begin{example}\label{ae}
 For a positive integer $s$, 
we consider the set 
\[
\mathcal{A}_{q,s}:=
\left\{\varphi(x) \in \mathbb{F}_q[x]\ \Big|\ \varphi(x)=\sum_{i=0}^n c_i x^{p^{si}}\right\}, 
\]
which is regarded as a ring with 
multiplication $\varphi_1 \circ \varphi_2:=\varphi_1(\varphi_2(x))$ for $\varphi_1,\varphi_2 \in \mathcal{A}_{q,s}$.

In the following, we give examples such that 
$E_R(x)$ is prime. 
We write  $d_R=p^t+1$ with $t \geq 0$. 
Then we have $E_R \in \mathcal{A}_{q,t}$. 
We write $q=p^f$. 
 Assume $f \mid t$. We have 
 \begin{equation}\label{ee}
 E_R(x)=\sum_{i=0}^e a_i x^{p^{e+i}}
 +\sum_{i=0}^e a_i x^{p^{e-i}}. 
 \end{equation}
By $f \mid  t$, we have the ring isomorphism 
$\Phi \colon \mathcal{A}_{q,t} \xrightarrow{\sim} \mathbb{F}_q[y];\ \sum_{i=0}^r c_i x^{p^{ti}} \mapsto \sum_{i=0}^r c_i y^i$, where 
$\mathbb{F}_q[y]$ is a usual polynomial ring. The polynomial $E_R(x) \in \mathscr{A}_q$ is prime if and only if $\Phi(E_R(x))$ is irreducible in $\mathbb{F}_q[y]$ in a usual sense. 
Recall that a polynomial $\sum_{i=0}^r 
c_i y^i \in \mathbb{F}_q[y]$ is said to be 
\textit{reciprocal} if $c_i=c_{r-i}$ for $0 \leq i \leq r$. 
By \eqref{ee}, we know that $\Phi(E_R(x))$ is a reciprocal polynomial. 
The number of the monic irreducible reciprocal 
polynomials is calculated in \cite[Theorems 2 and 3]{Ca1}. 

In general, we do not know a necessary and sufficient condition on $R(x)$
for $E_R(x)$ to be prime. 
 The number of prime elements in $\mathcal{A}_{q,s}$ is calculated in \cite{CHM} and \cite{O}. 
 \end{example}
 \begin{proposition}\label{last}
Assume $d_{R,m} \in \{1,2\}$. 
 There exists an unramified finite 
 extension $F'/F$ such that 
 $\tau_{\psi,R,m}|_{W_{F'}}$ is imprimitive. 
 \end{proposition}
 \begin{proof}
 For a positive integer $r$, let $F_r$ be the 
 unramified extension of $F$ of degree $r$. 
  We take a non-zero element $\beta \in V_R$. 
  Let $t$ be the positive integer such that 
 $\mathbb{F}_{q^t}=\mathbb{F}_q(\beta)$. 
 Let $\mathscr{H}_t \subset \mathscr{H}$
 be the subgroup generated by $\sigma^t,\tau$. 
 The subspace $W_R:=
 \mathbb{F}_p \beta \subset V_R$ is 
 a totally isotropic $\mathbb{F}_p[\mathscr{H}_t]$-submodule because of $d_{R,m} \leq 2$. Hence 
 $V_R$ is not a completely anisotropic   
 $\mathbb{F}_p[\mathscr{H}_t]$-module. Thus $\tau_{\psi,R,m}|_{W_{F_t}}$ is imprimitive by \cite[Theorem 4.1]{K}. 
 \end{proof}
  \begin{lemma}
 The $W_{T_{\rho}}$-representation 
 $\tau_{\psi,R,m}|_{W_{T_{\rho}}}$ is imprimitive. 
 \end{lemma}
 \begin{proof}
 We take a non-zero element $\beta \in V_R$. Then $\mathbb{F}_p \beta$ 
 is a totally isotropic symplectic submodule of the symplectic module $V_R$ associated to  
 $\tau_{\psi,R,m}|_{W_{T_{\rho}}}$. 
 Hence $\tau_{\psi,R,m}|_{W_{T_{\rho}}}$ is imprimitive by Corollary \ref{mc}(1). 
 \end{proof}
 \subsection{Root system associated 
 to irreducible $\mathbb{F}_p[\mathscr{H}]$-module}\label{root}
 A root system associated to  
 an irreducible $\mathbb{F}_p[\mathscr{H}]$-module is defined in \cite{K}. 
 We determine the root system associated to  
 $V_R$ in the situation of Corollary \ref{mc}(3). 
 
 We recall the definition of a root system. 
 \begin{definition}(\cite[\S7]{K})
 \begin{itemize}
\item[{\rm (1)}] Let $\Phi$ be the group of the automorphisms of the torus $ (\mathbb{F}^{\times})^2$ generated by the automorphisms 
 $\theta \colon (\alpha,\beta) \mapsto (\alpha^p,\beta^p)$ and 
 $\sigma \colon (\alpha,\beta)\mapsto (\alpha^{q^{-1}},\beta)$. 
 A $\Phi$-orbit of $ (\mathbb{F}^{\times})^2$ is called a \textit{root system}. 
 \item[{\rm (2)}] Let $W=\Phi(\alpha,\beta)$ 
 be a root system.  
 Let 
 \begin{align*}
& \textrm{$a=a(W)$ be the minimal positive integer 
with $\alpha^{q^{a}}=\alpha$}, \\
& \textrm{$b=b(W)$ the minimal positive integer 
with $\alpha^{p^{b}}=\alpha^{q^x}$, 
$\beta^{p^{b}}=\beta$ with $x \in \mathbb{Z}$, and } \\
& \textrm{$c=c(W)$ the minimal non-negative integer 
with $\alpha^{p^b}=\alpha^{q^{c}}$}. 
 \end{align*}
  Let $e'=e'(W)$ and $f'=f'(W)$ be the orders of $\alpha$ and $\beta$, respectively. 
  These integers are independent of $(\alpha,\beta)$ in $W$. 
 \item[{\rm (3)}] Let $\mathscr{H}_{d,r}:=\left\langle\sigma,\tau \mid \tau^d=1,\ \sigma^r=1,\ \sigma\tau\sigma^{-1}=\tau^q\right\rangle$ with $q^r \equiv 1 \pmod d$. 
\item[{\rm (4)}] 
We say that a root system $W$ \textit{belongs to 
 $\mathscr{H}_{d,r}$} if $e' \mid d$ and 
 $a f' \mid r$. 
 \item[{\rm (5)}] 
Let $W=\Phi (\alpha,\beta)$ be a root system which belongs to $\mathscr{H}_{d,r}$. 
Let $\overline{M(W)}$ be the $\mathbb{F}$-module with the basis 
\[
\{\theta^i \sigma^j m \mid 0 \leq i \leq b-1,\ 0 \leq j \leq a-1\}
\] 
and with the action of $\mathscr{H}$ by
\[
\tau m=\alpha m, \quad \sigma^a m =\beta m, \quad 
\theta^b m=\sigma^{-c} m. 
\]
 \end{itemize}
  \end{definition}
   \begin{theorem}(\cite[Theorems 7.1 and 7.2]{K})\label{at}
 \begin{itemize}
\item[{\rm (1)}] 
There exists an irreducible $\mathbb{F}_p[\mathscr{H}_{d,r}]$-module $M(W)$ 
such that $M(W) \otimes_{\mathbb{F}_p}\mathbb{F}$ is isomorphic to 
$\overline{M(W)}$ as $\mathbb{F}[\mathscr{H}_{d,r}]$-modules. 
\item[{\rm (2)}]
 The map $W \mapsto M(W)$ defines a one-to-one correspondence between the set of root systems belonging to $\mathscr{H}_{d,r}$
 and the set of isomorphism classes of irreducible 
 $\mathbb{F}_p[\mathscr{H}_{d,r}]$-modules. 
 \end{itemize}
 \end{theorem}
We go back to the original situation. 
Assume that $R(x)=a_e x^{p^e}$ and 
 $\mathbb{F}_p(\mu_{d_{R,m}})=
 \mathbb{F}_{p^{2e}}$. 
   Let $\mathscr{H}$ be 
   as in \eqref{hr}. In the above notation, we have 
   $\mathscr{H}=\mathscr{H}_{d_{R,m},r}$. 
 As in Corollary \ref{mc}(3),
 the $\mathbb{F}_p[\mathscr{H}]$-module 
 $V_R$ is irreducible. 
  \begin{proposition}\label{wpri}
We write  $q=p^f$. 
Let $e_1:=\mathrm{gcd}(f,2e)$ and $\beta:=\Nr_{q/p^{e_1}}(-a_e^{-(p^e-1)})$. 
Let $\alpha \in \mu_{d_{R,m}}$ be a primitive $d_{R,m}$-th root of unity.  We consider the root system $W:=\Phi(\alpha,\beta)$. 
\begin{itemize}
\item[{\rm (1)}] We have 
$
a(W)=2e/e_1$ and $b(W)=e_1. $
Further, 
$c(W)$ is the minimal non-negative integer 
such that $f c(W) \equiv e_1 \pmod{2e}$. 
\item[{\rm (2)}]  
The root system $W$ belongs to $\mathscr{H}$. 
\item[{\rm (3)}] 
We have an isomorphism $V_R \simeq M(W)$ as 
$\mathbb{F}_p[\mathscr{H}]$-modules. 
\end{itemize}
\end{proposition}
\begin{proof}
We show (1). 
We simply write  $a,b,c$ for $a(W), b(W),c(W)$, respectively. 
By definition, 
$a$ is the minimal natural integer 
such that $\alpha^{q^{a}}=\alpha$. 
By $\mathbb{F}_p(\alpha)=\mathbb{F}_{p^{2e}}$, $a$ is the minimal positive integer satisfying $f a \equiv 0 \pmod{2e}$. Thus we obtain $a
=2e/e_1$. 

By definition, $b$ is the minimal natural integer 
such that $\alpha^{p^b}=\alpha^{q^x}$ with some integer $x$ and $\beta^{p^b}=\beta$. 
The first condition implies that 
$fx \equiv b \pmod{2e}$. Hence 
$b$ is divisible by $e_1$. 
By $\beta \in \mathbb{F}_{p^{e_1}}^{\times}$,  
we have $\beta^{p^b}=\beta$ if $e_1 \mid b$. 
 Hence $b=e_1$. 
 
 By definition, $c$ is the minimal non-negative 
 integer such that $\alpha^{p^b}=\alpha^{q^c}$. 
This is equivalent to $e_1=b \equiv fc \pmod{2e}$.

 We show (2). The order $e'$ of $\alpha$ equals 
 $d_{R,m}$. 
 Let $f'$ be the order of $\beta$. It suffices to show $a f' \mid r$.  
 By the choice of $r$, we have $\alpha^{q^r}=\alpha$. 
 Hence $2e \mid fr$ by 
 $\mathbb{F}_{p^{2e}}=\mathbb{F}_p(\alpha)$, 
 and $a \mid r$. 
These imply that $\mathbb{F}_{p^{2e}} \subset \mathbb{F}_{q^a} \subset \mathbb{F}_{q^r}$.  

 Let $\eta \in V_R \setminus \{0\}$.  
By $\eta^{p^{2e}}=-a_e^{-(p^e-1)} \eta$, 
$a_e \in \mathbb{F}_q^{\times}$ and $2e \mid fr$, we compute 
\begin{equation}\label{frgd}
\eta^{q^r}=(\eta^{p^{2e}-1})^{\frac{q^r-1}{p^{2e}-1}}\eta
=\Nr_{q^r/p^{2e}}(-a_e^{-(p^e-1)}) \eta=\Nr_{{q^a}/{p^{2e}}}(-a_e^{-(p^e-1)})^{r/a} \eta. 
\end{equation}
The restriction map  
$\mathrm{Gal}(\mathbb{F}_{q^a}/\mathbb{F}_{p^{2e}})\to \mathrm{Gal}(\mathbb{F}_{q}/\mathbb{F}_{p^{e_1}})$ is an isomorphism because of $a=2e/e_1$. 
By $a_e \in \mathbb{F}_q^{\times}$, we have $\Nr_{{q^a}/{p^{2e}}}
(-a_e^{-(p^e-1)})=\beta$. 
Hence 
$\eta^{q^r}=\beta^{r/a} \eta$ by \eqref{frgd}. 
Since $\eta^{q^r}=\eta$ by Lemma \ref{sym0}, we obtain $\beta^{r/a}=1$. 
 Hence 
$f' \mid (r/a)$. 

We show (3).  
Let $\eta \in V_R \setminus \{0\}$. 
Similarly to \eqref{frgd}, we have 
$\sigma^a \eta=\eta^{q^a}=\beta \eta$. By definition and Lemma \ref{sym0}, we have 
$\tau \eta=\alpha \eta$. 
The $\mathbb{F}_p[\mathscr{H}]$-module $V_R$ satisfies the assumption in \cite[Lemma 7.3]{K} by (2). 
Hence $\{0\} \neq M(W) \subset V_R$
 by \cite[Lemma 7.3]{K}. 
By the irreducibility of $V_R$ in Corollary \ref{mc}(4), we obtain $M(W)=V_R$. 
\end{proof}
A necessary and sufficient condition for an irreducible $\mathbb{F}_p[\mathscr{H}]$-module to have a symplectic form is determined in 
\cite[Theorem 8.1]{K}. 
We recall the result. 
\begin{theorem}(\cite[Theorem 8.1]{K})\label{8.1}
Let $W=\Phi(\alpha,\beta)$ be a root system. 
The irreducible $\mathbb{F}_p[\mathscr{H}]$-module $M(W)$ has a symplectic form if and only if 
\begin{itemize}
\item[{\rm (A)}] $a(W) \equiv 0 \pmod{2}$, $\alpha \in \mu_{q^{a(W)/2}+1}$ and $\beta=-1$, 
\item[{\rm (B)}] $b(W),c(W) \equiv 0 \pmod{2}$, $\alpha \in 
\mu_{p^{b(W)/2}+q^{c(W)/2}}$ and $\beta \in \mu_{p^{b(W)/2}+1}$, or 
\item[{\rm (C)}] $b(W) \equiv 0 \pmod{2}$, $c(W) \equiv a(W) \pmod{2}$, $\alpha \in \mu_{p^{b(W)/2}+q^{(a(W)+c(W))/2}}$
and $\beta \in \mu_{p^{b(W)/2}+1}$. 
\end{itemize}
There are two isomorphism classes of symplectic 
structures on $M(W)$ in the case A, $p \neq 2$
and one in all other cases. 
\end{theorem}
\begin{lemma}\label{ggf}
Let $W$ be as in Proposition \ref{wpri}. 
Let $v_2(\cdot)$ denote the $2$-adic valuation on $\mathbb{Q}$. 
\begin{itemize}
\item[{\rm (1)}] 
Assume $v_2(e) \geq v_2(f)$.  
Then the module $M(W)$ is of type A in Theorem \ref{8.1}. 
\item[{\rm (2)}]
Assume $v_2(e)<v_2(f)$. Then we have $a(W) \equiv 1 \pmod{2}$, $b(W) \equiv 0 \pmod{2}$ and 
$(b(W)/2) \mid e$. 
Hence we have $\beta \in \mu_{p^{b(W)/2}+1}$. 
\begin{itemize}
\item[{\rm (i)}] If $c(W) \equiv 0 \pmod{2}$, 
the module $M(W)$ is of type B in Theorem \ref{8.1}. 
\item[{\rm (ii)}]  If $c(W) \equiv 1 \pmod{2}$, 
the module $M(W)$ is of type C in Theorem \ref{8.1}. 
\end{itemize}
\end{itemize}
\end{lemma}
\begin{proof}
We show (1). 
Recall that 
$e_1=\mathrm{gcd}(f,2e)$ and $\beta=\Nr_{q/p_1}(-a_e^{-(p^e-1)})$. 
We have $a(W)=2e/e_1 \equiv 0 \pmod{2}$. 
We have $e_1 \mid e$ and $f/e_1 \equiv 1 
\pmod{2}$. By $(p^{e_1}-1) \mid (p^e-1)$, 
\[
\beta=(-1)^{\frac{f}{e_1}}\bigl(a_e^{-\frac{p^e-1}{p^{e_1}-1}}\bigr)^{q-1}=-1, 
\] 
where we use $a_e \in \mathbb{F}_q^{\times}$
for the last equality. 
By $f a(W)/2=fe/e_1$ and $q=p^f$, we have 
$\alpha^{q^{a(W)/2}+1}=\alpha^{p^{fe/e_1}+1}$. Since $fe/e_1$ is divisible by
$e$ and $f/e_1$ is odd,  
$d_{R,m} \mid (p^e+1) \mid (p^{fe/e_1}+1)$. Hence  
we obtain $\alpha \in \mu_{q^{a(W)/2}+1}$. 
Thus the claim follows. 

We show (2). 
Recall $b(W)=e_1$.  
The former claims are 
clear. 
By $(e_1/2) \mid e$, we have 
$(p^{e_1/2}-1) \mid (p^e-1)$. 
By definition of $\beta$ and $a_e \in \mathbb{F}_q^{\times}$, we obtain  
\[
\beta^{p^{\frac{e_1}{2}}+1}=
\bigl(a_e^{-\frac{p^e-1}{p^{e_1/2}-1}}\bigr)^{q-1}=1. 
\]
Hence $\beta \in \mu_{p^{b(W)/2}+1}$. 
Assume that $c(W)$ is even. We write $(c(W)/2) f=(e_1/2)+l e$ with 
$l \in \mathbb{Z}$ by Proposition \ref{wpri}(1). 
Then $l$ is odd by $e_1=\mathrm{gcd}(f,2e)$. 
Hence $(p^e+1) \mid (p^{le}+1)$. This implies 
$\alpha \in \mu_{p^{b(W)/2}+q^{c(W)/2}}$. Hence 
we obtain (2)(i). 
The remaining claim is shown similarly. 
\end{proof}
\subsubsection{K\"unneth formula and primary module}
\paragraph{Classification results in \cite{K}}
We recall classification results on
completely anisotropic symplectic 
modules given in \cite{K} restricted to the case 
$p \neq 2$.  
\begin{theorem}(\cite[Theorem 9.1]{K})
Let 
$(V,\omega)
=\bigoplus_{i=1}^n(V_i,\omega_i)$ be a direct 
sum of irreducible symplectic 
$\mathbb{F}_p[\mathscr{H}]$-modules. 
Assume that $p \neq 2$. 
Then $(V,\omega)$ is completely anisotropic if and 
only if, for each isomorphism class, the modules of 
type B or C occur at most once and of type A
at most twice among $V_1,\ldots,V_n$.  
\end{theorem}
Assume that $p \neq 2$. 
Let $(M(W),0)$ denote the unique 
symplectic module
on $M(W)$ which is of type B or C by Theorem \ref{8.1}.  
Let $(M(W),0)$, $(M(W),1)$ denote the 
two symplectic modules on $M(W)$ 
 in the case where $p \neq 2$ and 
$M(W)$ is of type A.
We denote by $(M(W),2)$ the completely anisotropic symplectic module on $M(W) \oplus M(W)$, where $M(W)$ is of type A. 
\begin{theorem}(\cite[Theorem 8.2]{K})
Each completely anisotropic symplectic 
$\mathbb{F}_p[\mathscr{H}]$-module is 
isomorphic to one and only one symplectic 
module of the form
\[
\bigoplus_{i=1}^n(M(W_i),\nu_i), 
\]
where $W_1,\ldots, W_n$ are mutually 
different root systems belonging to $\mathscr{H}$. 
\end{theorem}

Let $k$ be a positive integer. 
Let $R:=\{R_i\}_{1 \leq i \leq k}$ 
with $R_i \in \mathscr{A}_q$. 
We consider the $k$-dimensional affine smooth variety $X_R$ defined by 
\[
a^p-a=\sum_{i=1}^k x_iR_i(x_i)
\]
in $\mathbb{A}_{\mathbb{F}_q}^{k+1}$. 
The product group $Q_R:=Q_{R_1} \times \cdots \times  Q_{R_k}$ 
acts on $X_R$ naturally similarly as \eqref{fa0}. 
Let $\mathbb{Z}$ act on $Q_R$ naturally.  
Let $\psi \in \mathbb{F}_p^{\vee} \setminus \{1\}$. 
We regard $H_{\rm c}^k(X_{R,\mathbb{F}},\overline{\mathbb{Q}}_{\ell})[\psi]$ as a $Q_R \rtimes \mathbb{Z}$-representation. 
Let the notation be as in \eqref{abcr}. 
Let $m=\{m_i\}_{1 \leq i \leq k}$, where $m_i$ is a positive integer. 
We have the homomorphism  
\begin{equation}\label{kkt}
\Theta_{R,m} \colon W_F \to Q_R \rtimes \mathbb{Z};\ \sigma \mapsto
((a_{R_i,\sigma}^{m_i},b_{R_i,\sigma},c_{R_i,\sigma})_{1 \leq i \leq k}, n_{\sigma}).  
\end{equation}
\begin{definition}
We define a smooth $W_F$-representation
 $\tau_{\psi,R,m}$ to be the 
inflation of the $Q_R \rtimes \mathbb{Z}$-representation $H_{\rm c}^k(X_{R,\mathbb{F}},\overline{\mathbb{Q}}_{\ell})[\psi]$ by 
$\Theta_{R,m}$. 
\end{definition}
\begin{lemma}\label{kk}
We have an isomorphism 
$\tau_{\psi,R,m} \simeq \bigotimes_{i=1}^k 
\tau_{\psi,R_i,m_i}$ as $W_F$-representations. 
\end{lemma}
\begin{proof}
Let $Q_{R_i,\mathbb{Z}}:
=Q_{R_i} \rtimes \mathbb{Z}$ and 
$\Theta_{R_i,m_i} \colon W_F \to Q_{R_i,\mathbb{Z}}$ be as in \eqref{tmr}. 
Let 
\[
\delta' \colon Q_R \rtimes \mathbb{Z} \to 
Q_{R_1,\mathbb{Z}} \times \cdots \times Q_{R_k,\mathbb{Z}};\  
((g_i)_{1 \leq i \leq k},n) \mapsto 
(g_i,n)_{1 \leq i \leq k}. 
\]
Each $H_{\rm c}^1(C_{R_i,\mathbb{F}},\overline{\mathbb{Q}}_{\ell})[\psi]$ is regarded as 
a $Q_{R_i,\mathbb{Z}}$-representation. 
By the K\"unneth formula, we have an isomorphism 
$H_{\rm c}^k(X_{R,\mathbb{F}},\overline{\mathbb{Q}}_{\ell})[\psi] \simeq 
\bigotimes_{i=1}^k (H_{\rm c}^1(C_{R_i,\mathbb{F}},\overline{\mathbb{Q}}_{\ell})[\psi])$ as 
$Q_R \rtimes \mathbb{Z}$-representations, 
where the right hand side is regarded as a 
$Q_R \rtimes \mathbb{Z}$-representation via $\delta'$. 
We consider the commutative diagram 
\[
\xymatrix{
W_F \ar[d]_{\Theta_{R,m}}\ar[r]^{\delta} & W_F^k \ar[d]^{\Theta_{R_1,m_1} \times  \cdots \times \Theta_{R_k,m_k}}\\
Q_R \rtimes \mathbb{Z} \ar[r]^-{\delta'} & Q_{R_1,\mathbb{Z}} \times \cdots \times Q_{R_k,\mathbb{Z}}, 
}
\]
where $\delta$ is the diagonal map. 
Hence the claim 
follows. 
\end{proof}
\begin{remark}
Let $+ \colon \prod_{i=1}^k Z(Q_{R_i}) \to \mathbb{F}_p;\ 
(1,0,\gamma_i)_{1 \leq i \leq k} \mapsto \sum_{i=1}^k\gamma_i$
and $\overline{Q}_R:=Q_R/\Ker +$. 
The action of $Q_R \rtimes \mathbb{Z}$
on $H_{\rm c}^k(X_{R,\mathbb{F}},\overline{\mathbb{Q}}_{\ell})$
factors through $\overline{Q}_R \rtimes 
\mathbb{Z}$. 
Let $\overline{H}_R$ denote the image of 
$H_{R_1} \times \cdots \times H_{R_k}$ 
under $Q_R \to \overline{Q}_R$.  
The group $\overline{H}_R$ is an extra-special $p$-group. 
The quotient $\overline{H}_R/Z(\overline{H}_R)$
is isomorphic to $\bigoplus_{i=1}^k V_{R_i}$. 
Moreover, $\overline{Q}_R/\overline{H}_R$ is 
supersolvable. 
\end{remark}
\begin{lemma}
The $W_F$-representation 
$\tau_{\psi,R,m}$ is irreducible.
\end{lemma}
\begin{proof}
The $\overline{H}_R$-representation 
$H_{\rm c}^k(X_{R,\mathbb{F}},\overline{\mathbb{Q}}_{\ell})[\psi]$ is 
irreducible by \cite[16.14(2) Satz]{H}. The claim follows from this. 
\end{proof}
Let 
$\rho_{\psi,R_i,m_i}$ denote the projective representation associated to 
$\tau_{\psi,R_i,m_i}$. 
Let $F_i$ denote the kernel field of $\rho_{\psi,R_i,m_i}$ and $T_i$ the 
maximal tamely ramified extension of $F$ in $F_i$. The field $T_i$ is called the tame kernel field 
of $\rho_{\psi,R_i,m_i}$.   
Let $F_R:=F_1\cdots F_k$.
\begin{lemma}
Let 
$\rho_{\psi,R,m}$ be the projective representation associated to 
$\tau_{\psi,R,m}$. 
The kernel field of 
$\rho_{\psi,R,m}$ is $F_R$. 
\end{lemma}
\begin{proof}
By Lemma \ref{kk}, we can check $\Ker \rho_{\psi,R,m}=\bigcap_{i=1}^k \Ker \rho_{\psi,R_i,m_i}$. 
The claim follows from this. 
\end{proof}
Let $T_R$ be the maximal tamely
ramified extension of $F$ in $F_R$. 
We have the restriction map 
$V_R \hookrightarrow \prod_{i=1}^k \Gal(F_i/T_i)
\simeq \bigoplus_{i=1}^k V_{R_i}$. 
Then $V_R:=\Gal(F_R/T_R)$ has a bilinear form stable under the action of $\mathbb{F}_p[\Gal(T_R/F)]$ 
(\cite[\S4]{K}). The form on $V_R$ is given by
$\omega_R:=\sum_{i=1}^k \omega_{R_i}$.

Let $\omega_{R_i}$ be the form on 
$V_{R_i}$ in Lemma \ref{ab}(2). 
We give a recipe to make an example of $(M(W),2)$ below. 
\begin{proposition}\label{pree}
Assume $k=2$. 
Let $R_i(x)=a_{e,i}x^{p^e} \neq 0$ for 
$i \in \{1,2\}$.
Assume 
\[
m_1 \neq m_2, \quad  d:=d_{R_1,m_1}=d_{R_2,m_2}. 
\]
\begin{itemize}
\item[{\rm (1)}] We have an isomorphism 
$V_R \simeq V_{R_1} \oplus V_{R_2}$.  
\item[{\rm (2)}] We have $T_R=T_1 \cdot T_2$. 
\item[{\rm (3)}] Assume that $p \neq 2$, $v_2(e) \geq v_2(f)$ and $\mathbb{F}_p(\mu_d)=\mathbb{F}_{p^{2e}}$. 
If $(V_R,\omega_R)$ is completely anisotropic as a symplectic $\mathbb{F}_p
[\Gal(T_R/F)]$-module, 
$V_R$ is isomorphic to a primary module $(M(W),2)$ with a root system $W$. 
\end{itemize}
\end{proposition}
\begin{proof}
By Lemma \ref{inj} and Lemma \ref{rem1}, 
there exists an unramified finite extension $E$
of $F$ such that $F_i \subset E(\alpha_{R_i}^{m_i},\beta_{R_i,m_i})$ for $i=1,2$ and 
$E(\alpha_{R_i}^{m_i},\beta_{R_i,m_i})/E$ is Galois.  
We put $T:=E(\alpha_{R_i}^{m_i})=E(\varpi^{1/d})$ and 
$E_i:=T(\beta_{R_i,m_i})$ for 
$i=1,2$.  
Let $n_i:=m_i d/d_R=m_i/\mathrm{gcd}(d_{R},m_i)$. 
Let $\{\Gal(E_i/T)^v\}_{v \geq -1}$ be the upper numbering ramification subgroups of $\Gal(E_i/T)$. 
Similarly as the proof of Lemma \ref{her}, we have  
\[
\Gal(E_i/T)^v=\begin{cases}
\Gal(E_i/T) & \textrm{if $v \leq n_i$}, \\
\{1\} & \textrm{if $v > n_i$}. 
\end{cases}
\] 
Let $H:=E_1\cap E_2$. 
Since $E_i/T$ is Galois, so is $H/T$. 
By \cite[Proposition 14 in IV\S3]{Se}, the subgroup $\Gal(H/T)^v$ equals $\Gal(H/T)$ if 
$v \leq n_i$ and $\{1\}$ if $v>n_i$. 
Hence we conclude $\Gal(H/T)=\{1\}$ by $n_1 \neq n_2$. 
We obtain $H=T$. Hence we have an isomorphism
$\Gal(E_1E_2/T) \simeq \Gal(E_1/T) \times 
\Gal(E_2/T) \simeq V_{R_1} \oplus V_{R_2}$. 
The extension $E_1E_2/T$ is totally ramified 
and 
the degree is $p$-power. Hence, $T$ is the maximal tamely ramified extension of $E$ in $E_1 \cdot E_2$. Therefore, $T_R=F_R \cap T$. 
We have the commutative diagram 
\[
\xymatrix{
\Gal(E_1 E_2/T) \ar[d]\ar[r]^-{\simeq} & \Gal(E_1/T) \times 
\Gal(E_2/T) \ar[d]^{\simeq} \\
\Gal(F_R/T_R) \ar[r]^-g & \Gal(F_1/T_1) \times 
\Gal(F_2/T_2),}  
\] 
where every map is the restriction map. The right vertical isomorphism follows from Lemma \ref{asso}. 
Clearly $g$ is injective. 
By the commutative diagram, 
$g$ is bijective. 
Hence we obtain (1). 

By the commutative diagram 
\[
\xymatrix
{1 \ar[r] & \Gal(F_R/T_R) \ar[r]\ar[d]^{\simeq} & \Gal(F_R/F) \ar[r]\ar[d]^{g_1} & \Gal(T_R/F) \ar[r]\ar[d]^{g_2} & 1\\ 
1 \ar[r]& V_{R_1}\oplus V_{R_2}  \ar[r] & \Gal(F_1/F) \times \Gal(F_2/F) \ar[r] & \Gal(T_1/F) \times \Gal(T_2/F) \ar[r] & 1}
\]
and injectivity of $g_1$, 
the map $g_2$ is injective. Hence 
$T_R=T_1 T_2$.

We show (3). 
Let $r:=[E:F]$ and  
$\mathscr{H}_{d,r}:=\mu_d \rtimes 
(\mathbb{Z}/r\mathbb{Z})$ as in \eqref{hr}.  
We identify $\Gal(T/F)$ with $\mathscr{H}_{d,r}$. 
By $T_R \subset T$, the 
$V_R$, $V_{R_i}$ are naturally regarded as 
$\mathbb{F}_p[\mathscr{H}_{d,r}]$-modules. 
Let $\alpha$ be a primitive $d$-th root of unity. Let $W:=\Phi(\alpha,-1)$. 
Then we have an isomorphism 
$V_{R_i} \simeq M(W)$ as 
$\mathbb{F}_p[\mathscr{H}_{d,r}]$-modules
and know that $V_{R_i}$ is of type A
by Proposition \ref{wpri}(3), Lemma \ref{ggf}(1) and $d_{R_1,m_1}=d_{R_2,m_2}$. 
This implies an isomorphism 
$V_{R_1} \simeq V_{R_2}$ as 
$\mathbb{F}_p[\mathscr{H}_{d,r}]$-modules. 
Hence the claim follows from the assumption that  $(V_R,\omega_R)$ is completely anisotropic
and the definition of $(M(W),2)$. 
\end{proof}
\begin{example}
Assume $p \neq 2$. 
Let $e=f=1$, $R_1(x)=x^p$ and 
$R_2(x)=a x^p \in \mathbb{F}_p[x] \setminus \{0\}$. 
We assume that 
$m_1 \neq m_2$ and $d_{R_1,m_1}
=d_{R_2,m_2}=p+1$. 
We have $V_{R_i}=\{x \in \mathbb{F} \mid 
x^{p^2}+x=0\}$ for $i=1,2$. 

Let $W \subset V_{R_1} \oplus V_{R_2}$
be a totally isotropic $\mathbb{F}_p[\Gal(T_R/F)]$-subspace. Assume $W \neq \{0\}$. We take a non-zero element $(x_1,x_2) \in W$. We have 
$f_{R_1}(x,y)=-xy^p$, $f_{R_2}(x,y)=-a xy^p$ and 
hence 
$\omega_R((x_1,x_2),(\xi x_1,\xi x_2))=
(x_1^{p+1}+a x_2^{p+1})(\xi-\xi^p)=0$ for 
any $\xi \in \mu_{p+1}$. Hence
$x_1^{p+1}+a x_2^{p+1}=0$ and $x_2 \neq 0$. There exists 
$\eta \in \mathbb{F}$ such that $\eta^{p+1}=-a$
and $x_1=\eta x_2$. 
By $\mathbb{F}_{p^2}=\mathbb{F}_p(\mu_{p+1})$,  we have $W_1:=\{(\eta x,x) \mid x \in V_{R_2}\}
\subset W$. We also have 
$W_2:=\{(\eta^p x, x) \mid x \in V_{R_2}\} \subset W$. 
Let $\bigl(\frac{\cdot}{p}\bigr)$ be the Legendre symbol. 
If $W_1 \cap W_2 \neq \{0\}$, we have $\eta \in \mathbb{F}_p$ and 
$\eta^2=-a$. 
This implies $\bigl(\frac{-a}{p}\bigr)=1$. 

Assume $\bigl(\frac{-a}{p}\bigr)=-1$. Then 
$W=W_1 \oplus W_2=V_{R_1} \oplus V_{R_2}$
 by $W_1 \cap W_2=\{0\}$. 
This is a contradiction. 
Hence $V_{R_1} \oplus V_{R_2}$ is 
completely anisotropic if $\bigl(\frac{-a}{p}\bigr)=-1$. 

If $\bigl(\frac{-a}{p}\bigr)=1$, we have $W_1=W_2$, which is the unique non-zero 
totally isotropic $\mathbb{F}_p[\mathscr{H}]$-subspace. Hence $V_{R_1} \oplus V_{R_2}$
is not completely anisotropic. 
\end{example}
 \section{Geometric interpretation of imprimitivity}\label{4}
 Through this section, we always assume $p \neq 2$. 
 Our aim in this section is to show Theorem \ref{4mm}. 
To show the theorem, 
we use the explicit understanding of the automorphism 
group of $C_R$ and the mechanism of 
taking quotients of $C_R$ by certain abelian groups, which are developed in \cite{BHMSSV} and \cite{GV}. 
\subsection{Quotient of $C_R$ and 
description of $\tau_{\psi,R,m}$}\label{4.1}
 Let $C_R$ be as in \eqref{cr}. 
In this subsection, we always assume that 
there exists a finite \'etale morphism 
\begin{equation}\label{aphi}
\phi \colon C_R \to C_{R_1};\ (a,x) \mapsto (a-\Delta(x),r(x)), 
\end{equation}
where $\Delta(x) \in \mathbb{F}_q[x]$ and 
$r(x), R_1(x) \in \mathscr{A}_q$ satisfy 
$d_{R,m} \mid d_{R_1}$ and 
$r(\alpha x)=\alpha r(x)$ for 
$\alpha \in \mu_{d_{R,m}}$. 
Since $\phi$ is \'etale, 
 $r(x)$ is a reduced polynomial. 
 Hence $r'(0)\neq 0$. 
By the above assumption, 
\begin{align}\label{(a)}
& xR(x)=r(x)R_1(r(x))
 +\Delta(x)^p-\Delta(x)
\\ \label{(b)} 
& r'(0)\neq 0, \quad  
d_{R,m} \mid d_{R_1},\quad 
r(\alpha x)=\alpha r(x) \quad \textrm{for 
 $\alpha \in \mu_{d_{R,m}}$}.
 \end{align} 
 Let 
$e'$ be a non-negative integer such that 
 $\deg R_1(x)=p^{e'}$ and $e' \leq e$. 
 Then $\deg r(x)=p^{e-e'}$ by \eqref{(a)}.  
 
 We have $\alpha R_1(\alpha x)=R_1(x)$
 for $\alpha \in \mu_{d_{R,m}}$ by 
 $d_{R,m} \mid d_{R_1}$ and \eqref{12}. 
 Hence $\Delta(\alpha x)-\Delta(x) \in \mathbb{F}_p$ for 
 $\alpha \in \mu_{d_{R,m}}$ by \eqref{(a)}.  
 We have $\Delta(\alpha x)=\Delta(x)$, since 
 the constant coefficient of $\Delta(\alpha x)-\Delta(x)$ is zero.  
 
 \begin{lemma}\label{triv}
 Let $\phi \colon C_R \to C_R;\ 
 (x,a) \mapsto (x+c,a+g(x))$ be the automorphism 
 with $g(x) \in \mathbb{F}_q[x]$ and 
 $c \in \mathbb{F}$. Then we have $E_R(c)=0$. 
 \end{lemma}
 \begin{proof}
 We have $g(x)^p-g(x)=cR(x)+xR(c)+cR(c)$. 
 Let $\mathcal{P} \colon 
 \mathbb{F}[x] \to  \mathbb{F}[x];\ 
 f(x) \mapsto f(x)^p-f(x)$. 
By the definition of $E_R(x)$, we obtain $cR(x)+xR(c)+cR(c) \equiv E_R(c)^{1/p^e}x
 \mod \mathcal{P}(\mathbb{F}[x])$. 
 Thus we must have $E_R(c)=0$. 
 \end{proof}
 \begin{lemma}\label{bb2}
 We have $E_{R_1}(r(x)) \mid E_R(x)$. 
 \end{lemma} 
 \begin{proof}
 Let $\beta \in \mathbb{F}$ be an element 
 such that $E_{R_1}(r(\beta))=0$. We take an element  
 $\gamma \in \mathbb{F}$  
 such that $\gamma^p-\gamma=r(\beta) R_1(r(\beta))$. 
 The curve $C_{R,\mathbb{F}}$ admits the automorphism $\phi$
 defined by
\[
\phi(a, x)=\left(a+f_{R_1}(r(x),r(\beta))+\Delta(x+\beta)-\Delta(x)+\gamma, x+\beta\right).  
\]
This is well-defined by Lemma \ref{a} and \eqref{(a)}. 
By Lemma \ref{triv}, 
we have $E_R(\beta)=0$. 
Since $E_{R_1}(r(x))$ is separable, 
the claim follows. 
 \end{proof}
 \begin{lemma}\label{bb}
 Let $\alpha,\alpha' \in \mu_{d_{R,m}}$. 
 Assume $E_{R_1}(r(\alpha y))=0$ for a 
 certain $y \in \mathbb{F}$. 
 Then we have the equality
 \[
 \Delta(\alpha' x+\alpha y)+f_{R_1}(r(\alpha' x),r(\alpha y))=
 \Delta(x)
 +\Delta(y)+f_R(\alpha' x,\alpha y). 
 \] 
 \end{lemma}
 \begin{proof}
By $\Delta(\alpha' x+\alpha y)
 =\Delta(x+(\alpha/\alpha') y)$ and \eqref{a-1},
  we may assume 
 $\alpha'=1$ by \eqref{(b)}. We have 
 $E_R(\alpha y)=0$ by Lemma \ref{bb2}. 
 Let $\Delta_1(x)$ and $\Delta_2(x)$
 denote the left and right hand sides of the required equality, respectively. We have $\Delta_1(0)=\Delta(\alpha y)=\Delta(y)=\Delta_2(0)$, since 
 $f_R(0,x') \equiv 0$ in $\mathbb{F}_q[x']$ by definition. 
Hence it suffices to show 
$\Delta_1(x)^p-\Delta_1(x)=\Delta_2(x)^p-\Delta_2(x)$.  
By Lemma \ref{bb2} and the assumption, 
$E_{R_1}(r(\alpha y))=E_R(\alpha y)=0$. 
Hence each $\Delta_i(x)^p-\Delta_i(x)$ for 
$i=1,2$ equals 
$(x+\alpha y)R(x+\alpha y)-r(y) R_1(r(y))-r(x) R_1(r(x))$ according to Lemma \ref{a}. 
 Hence the claim follows. 
 \end{proof}
 Let 
\[
U_R:=\{x \in \mathbb{F} \mid 
r(x)=0\} \subset V'_R:=\{x \in \mathbb{F}
\mid E_{R_1}(r(x))=0\}.  
\]
We have $V'_R \subset V_R$ by Lemma \ref{bb2}. 
Then $U_R$ and $V'_R$ are regarded as $\mathbb{F}_p[\mathscr{H}]$-modules by $r(x), R_1(x) \in \mathbb{F}_q[x]$ and \eqref{(b)}.  
\begin{lemma}\label{tiso}
We have $V'_R \subset 
U^{\perp}_R$. In particular, 
the $\mathbb{F}_p[\mathscr{H}]$-module $U_R$ is totally isotropic. 
\end{lemma}
\begin{proof}
Let $\beta \in U_R$ and $\beta' \in V'_R$. 
By Lemma \ref{bb}, 
$r(\beta)=0$ and  $E_{R_1}(r(\beta'))=0$, we have 
$f_R(\beta',\beta)=f_R(\beta,\beta')
=\Delta(\beta+\beta')-\Delta(\beta)-\Delta(\beta')$. Hence 
$\omega_R(\beta,\beta')=0$. 
\end{proof}
Let 
\[
Q'_{R,m}:=\{(\alpha,\beta,\gamma) \in Q_{R,m}
\mid \beta \in V'_R\}. 
\]
Then $Q'_{R,m}$ is a subgroup of $Q_{R,m}$ of index $p^{e-e'}$, because of \eqref{(b)} and $[V_R:V'_R]=p^{e-e'}$. 
We have the map
\[
\pi \colon Q'_{R,m} \to Q_{R_1,m};\ (\alpha,\beta,\gamma) \mapsto 
(\alpha,r(\beta),\gamma-\Delta(\beta)).
\]
 \begin{corollary}
 The map $\pi$ is a homomorphism. 
 \end{corollary}
 \begin{proof}
 The claim follows from Lemma \ref{bb} and \eqref{(b)}. 
 \end{proof}
 We have 
 \begin{equation}\label{fol}
 U'_R:=\{(1,\beta,\Delta(\beta))\in Q'_{R,m} \mid \beta \in U_R\}=\Ker \pi.
 \end{equation} 
 The space $V'_R$ is stable by the 
 $q$-th power map. Hence we can consider 
 the semidirect product 
 $Q'_{R,m} \rtimes \mathbb{Z}$. 
 The map $\pi$ induces 
 $\pi' \colon Q'_{R,m} \rtimes \mathbb{Z}
 \to Q_{R_1,m} \rtimes \mathbb
{Z}$. 
 \paragraph{Quotient of $C_R$}
Let $\phi$ be as in \eqref{aphi}. 
 We can check that $\phi$ factors through 
 $C_{R,\mathbb{F}} \to C_{R,\mathbb{F}}/U'_R \xrightarrow{\bar{\phi}} C_{R_1,\mathbb{F}}$ by \eqref{fa0}. 
We obtain an isomorphism $\bar{\phi} \colon C_{R,\mathbb{F}}/U'_R \xrightarrow{\sim} C_{R_1,\mathbb{F}}$. 
\begin{lemma}\label{q}
We have $\phi((a,x) g)=\phi(a,x) \pi'(g)$
for $g \in Q'_{R,m} \rtimes \mathbb{Z}$. 
\end{lemma}
\begin{proof}
The claim follows from Lemma \ref{bb}. 
\end{proof}
 Let $\tau'_{\psi,R_1,m}$ denote the 
 $Q'_{R,m} \rtimes \mathbb{Z}$-representation 
 which is the inflation of 
 the $Q_{R_1,m} \rtimes \mathbb{Z}$-representation 
 $H_{\rm c}^1(C_{R_1,\mathbb{F}},\overline{\mathbb{Q}}_{\ell})[\psi]$ by 
 $\pi'$. 
 By \eqref{tmr}, we have the homomorphism 
 $\Theta_{R,m} \colon W_F \to Q_{R,m} \rtimes 
 \mathbb{Z}$. 
We define a $W_F$-representation $\tau''_{\psi,R_1,m}$
 to be the inflation of 
 $\mathrm{Ind}_{Q'_{R,m} \rtimes \mathbb{Z}}
^{Q_{R,m} \rtimes \mathbb{Z}} \tau'_{\psi,R_1,m}$
via $\Theta_{R,m}$. 
We have $\dim \tau''_{\psi,R_1,m}
=p^e$ by $[Q_{R,m}:Q'_{R,m}]=p^{e-e'}$ and
  $\dim \tau'_{\psi,R_1,m}=p^{e'}$.
\begin{proposition}\label{qq}
We have an isomorphism 
$\tau_{\psi,R,m} \simeq \tau''_{\psi,R_1,m}$ as $W_F$-representations. 
\end{proposition}
\begin{proof}
By Lemma \ref{q}, we have the injection
\[
\tau'_{\psi,R_1,m}=H_{\rm c}^1(C_{R_1,\mathbb{F}},\overline{\mathbb{Q}}_{\ell})[\psi] \xrightarrow{\phi^\ast} H_{\rm c}^1(C_{R,\mathbb{F}},\overline{\mathbb{Q}}_{\ell})[\psi]
\] 
as $Q'_{R,m} \rtimes \mathbb{Z}$-representations. 
Hence we have a non-zero homomorphism 
\begin{equation}\label{indo}
\mathrm{Ind}_{Q'_{R,m} \rtimes \mathbb{Z}}
^{Q_{R,m} \rtimes \mathbb{Z}} \tau'_{\psi,R_1,m}
\to H_{\rm c}^1(C_{R,\mathbb{F}},\overline{\mathbb{Q}}_{\ell})[\psi]
\end{equation} 
as $Q_{R,m} \rtimes \mathbb{Z}$-representations 
by Frobenius reciprocity. 
Since the target is irreducible by Lemma \ref{fa}(2), 
the map \eqref{indo} is surjective. By comparing the dimensions, \eqref{indo} is an isomorphism. By inflating it by $\Theta_{R,m}$, we obtain the claim.  
\end{proof}
We consider the open subgroup 
$W':=\Theta_{R,m}^{-1}(Q'_{R,m} \rtimes \mathbb{Z}) \subset W_F$ of index $p^{e-e'}$. We can write 
$W'=W_{F'}$ with a finite field extension $F'/F$ of degree $p^{e-e'}$. 
Let 
\begin{equation}\label{taud}
\tau'_{\psi,R_1,m} \colon W_{F'} \xrightarrow{\Theta_{R,m}} Q'_{R,m} \rtimes \mathbb{Z} \xrightarrow{\pi'} Q_{R_1,m} \rtimes 
\mathbb{Z} \to \mathrm{Aut}_{\overline{\mathbb{Q}}_{\ell}}(H_{\rm c}^1(C_{R_1,\mathbb{F}},\overline{\mathbb{Q}}_{\ell})[\psi])
\end{equation}
 be the composite.

\begin{corollary}\label{mc4}
We have an isomorphism 
$\tau_{\psi,R,m} \simeq \Ind_{W_{F'}}^{W_F}\tau'_{\psi,R_1,m}$ as $W_F$-representations.  
If $e'<e$, the $W_F$-representation 
$\tau_{\psi,R,m}$ is imprimitive. 
\end{corollary}
\begin{proof}
The assertion follows from Proposition \ref{qq}. 
\end{proof}
 \subsection{Totally isotropic subspace and geometry of $C_R$}\label{4.2}
Let $(1,\beta,\gamma) \in H_R$ so, as in Definition \ref{qdef}(2), we know that $\gamma^p-\gamma=\beta R(\beta)$.
We obtain  
$(f_R(\beta,\beta)-2\gamma)^p=f_R(\beta,\beta)-2 \gamma$ by definition of the pairing $\omega_R$ (Lemma \ref{ab}(2)). Assume 
\begin{equation}\label{ad}
\beta \neq 0, \quad 
\gamma=\frac{f_R(\beta,\beta)}{2}.
\end{equation} 
The following lemma is given in \cite[Propositions (9.1) and (13.5)]{GV} and \cite[Proposition 7.2]{BHMSSV}. This lemma gives an algorithm of taking quotients of $C_R$ by certain abelian groups. 
\begin{lemma}\label{q0}
Let $C_R$ be as in Definition \ref{cr}. Assume $e \geq 1$. 
\begin{itemize}
\item[{\rm (1)}] Let 
\begin{equation}\label{ch}
u:=x^p-\beta^{p-1}x, \quad 
v:=a+(x/\beta)(\gamma(x/\beta)-f_R(x,\beta)). 
\end{equation}
Then there exists 
$P_1(u) \in \mathscr{A}_{\mathbb{F}}$ of degree $p^{e-1}$ such that 
$v^p-v=u P_1(u)$. 
\item[{\rm (2)}] 
Let $U:=\{(1,\xi \beta,\xi^2 \gamma) \in 
H_R \mid \xi \in \mathbb{F}_p\}=\langle (1,\beta,\gamma)\rangle$. 
Then the quotient $C_{R,\mathbb{F}}/U$ is isomorphic to $C_{P_1,\mathbb{F}}$.  
\end{itemize}
\end{lemma}
\begin{proof}
We show (1). 
Let $x_1:=x/\beta$ and 
$u_1:=u/\beta^p$. Then 
$u_1=x_1^p-x_1$. 
We compute 
\begin{align*}
v^p-v&=
xR(x)
+\gamma^p x_1^{2p}-\gamma x_1^2-x_1^p
{f_R(x,\beta)}^p+x_1 f_R(x,\beta) \\
&=xR(x)+\gamma(x_1^{2p}-x_1^2)+
\beta^{-2p+1} R(\beta) x^{2p}\\
&\ \ -u_1 f_R(x,\beta)-(x/\beta)^p(\beta R(x)+x R(\beta)) \\
&=u \beta^{-p} (-\beta R(x)+\beta^{-p+1}
R(\beta) x^p+\gamma (x_1^p+x_1)-f_R(x,\beta)), 
\end{align*}
where we use $\gamma^p-\gamma=\beta R(\beta)$ and Lemma \ref{a} for the second equality. 
Let $P(x):=\beta^{-p}(-\beta R(x)+\beta^{-p+1}
R(\beta) x^p+\gamma (x_1^p+x_1)-f_R(x,\beta))$. 
Since $P(x)$ is additive, there exists  
$P_1(u) \in \mathscr{A}_{\mathbb{F}}$ such that 
$P(x)=P_1(u)+\alpha x$ with a constant $\alpha$. By \eqref{ad}, we have $P(\beta)=\beta^{-p}(2\gamma-f_R(\beta,\beta))=0$. 
Hence $\alpha=0$. 
By $\deg P(x)=p^e$, we have $\deg P_1(u)=p^{e-1}$. 
Hence we obtain (1). 

We show (2). We easily check that the finite \'etale morphism of degree $p$:  
$C_{R,\mathbb{F}} \to C_{P_1,\mathbb{F}};\ (a,x) \mapsto 
(v,u)$ factors 
through $C_{R,\mathbb{F}} \to C_{R,\mathbb{F}}/U \to C_{P_1,\mathbb{F}}$. 
Hence we obtain the claim. Since $C_{R,\mathbb{F}} \to C_{R,\mathbb{F}}/U$ is a finite \'etale morphism of degree $p$, 
the claim follows. 
\end{proof}
Let 
\[
\Delta_0(x):=-(x/\beta)(\gamma(x/\beta)-
f_R(x,\beta)).
\] 
We have 
\begin{equation}\label{cd}
xR(x)=u P_1(u)+\Delta_0(x)^p-\Delta_0(x)
\end{equation}
 by Lemma \ref{q0}(1). We write $u(x)$ for $u$. 

Let $(1,\beta',\gamma') \in H_R$ 
be an element satisfying \eqref{ad}.
Assume  $\omega_R
(\beta,\beta')=0$. Then $(1,\beta,\gamma)$ 
 commutes with $(1,\beta',\gamma')$. Hence the action of $(1,\beta',\gamma')$ induces the automorphism of $C_{P_1,\mathbb{F}} \simeq C_{R,\mathbb{F}}/U$ (\eqref{fa0}). 
\begin{lemma}\label{q2}
Let $\pi(\beta',\gamma'):=(1,u(\beta'),\gamma'-\Delta_0(\beta'))$. 
\begin{itemize}
\item[{\rm (1)}] We have $\pi(\beta',\gamma') \in H_{P_1}$ and 
 $f_{P_1}(u(\beta'),u(\beta'))=2(\gamma'-\Delta_0(\beta'))$. 
 \item[{\rm (2)}] The action of 
$(1,\beta',\gamma')$ on $C_{R,\mathbb{F}}$ induces $\pi(\beta',\gamma')$ on  
$C_{P_1,\mathbb{F}}$.
\end{itemize}
\end{lemma}
\begin{proof}
Let $\Delta_1(x):=f_R(x,\beta')-\Delta_0(x+\beta')+\Delta_0(x)$. 
By \eqref{ch}, the action of $(1,\beta',\gamma')$ on $C_{R,\mathbb{F}}$ induces the automorphism of $C_{P_1,\mathbb{F}}$ given by  
$u \mapsto u+u(\beta')$
and $v \mapsto v+\Delta_1(x)+\gamma'$ on $C_{P_1,\mathbb{F}}$. 
We can easily 
check that $\Delta_1(x)-\Delta_1(0)$
is an additive polynomial such that $\Delta_1(\beta)-\Delta_1(0)=\omega_R(\beta,\beta')=0$. 
Hence there exists $g(u) \in \mathbb{F}_q[u]$
such that $\Delta_1(x)=g(u(x))+\Delta_1(0)$. 
Lemma \ref{triv} implies that $E_{P_1}(u(\beta'))=0$. Hence $u(\beta') \in V_{P_1}$.  
We show (1). 
The former claim follows from \eqref{cd}. 
By using $\Delta_0(0)=E_{P_1}(u(\beta'))=E_R(\beta')=0$ in the same way as Lemma \ref{bb}, we have 
\begin{equation}\label{cd2}
\Delta_0(x+\beta')+f_{P_1}(u(x),u(\beta'))
=\Delta_0(x)+\Delta_0(\beta')+f_R(x,\beta'). 
\end{equation}
Substituting $x=\beta'$, and using 
$\Delta_0(2 \beta')=4 \Delta_0(\beta')$ and \eqref{ad} for $(\beta',\gamma')$, we obtain the latter claim in (1). 

By \eqref{cd2}, we have 
\[
v+f_R(x,\beta')-\Delta_0(x+\beta')+\Delta_0(x)+\gamma'=v+f_{P_1}(u(x),u(\beta'))+\gamma'-\Delta_0(\beta'). 
\]
Hence the claim (2) follows from \eqref{fa0}. 
\end{proof}
Assume that $V_R$ is not completely anisotropic. 
Let $U_R$ be a non-zero totally isotropic 
$\mathbb{F}_p[\mathscr{H}]$-submodule in $V_R$. 
There exists a monic reduced polynomial 
$r(x) \in \mathscr{A}_{\mathbb{F}}$ such that $U_R=\{x \in\mathbb{F} \mid r(x)=0\}$ by \cite[Theorem 7]{Ore}. Since $U_R$ is an $\mathbb{F}_p[\mathscr{H}]$-module, we have 
\begin{equation}\label{r0r}
\textrm{$r(\alpha x)=\alpha r(x)$ for $\alpha \in \mu_{d_{R,m}}$ and $r(x) \in \mathbb{F}_q[x]$}
\end{equation}
  by Lemma \ref{ele}. 
We write $\deg r(x)=p^{e-e'}$ with a non-negative integer $0 \leq e' < e$.

We take a basis $\beta_1,\ldots,\beta_{e-e'}$
of $U_R$ over $\mathbb{F}_p$.  
Let $(1,\beta_i,\gamma_i) \in H_R$ 
be an element which satisfies \eqref{ad}. 
Let $U_i:=
\{(1,\xi \beta_i,\xi^2 \gamma_i) \mid \xi \in \mathbb{F}_p\} \subset H_R$, which is a subgroup. Since $U_R$ is totally isotropic, 
we have  $\omega_R(\beta_i,\beta_j)=0$.
Thus $g_i g_j=g_j g_i$ for any $g_i \in U_i$ and 
$g_j \in U_j$ by Lemma \ref{ab}(2).
Let 
\begin{equation}\label{fol1}
U'_R:=U_1\cdots U_{e-e'}
\subset H_R,  
\end{equation}
which is an abelian subgroup. 
\begin{proposition}\label{rp}
Assume that $V_R$ is not completely anisotropic. 
Then there exist 
$R_1(x) \in \mathscr{A}_{\mathbb{F}}$ of degree $p^{e'}$ and 
a polynomial $\Delta(x) \in \mathbb{F}[x]$ such that $\Delta(0)=0$ and 
the quotient $C_{R,\mathbb{F}}/U'_R$ is isomorphic to
the affine curve $C_{R_1,\mathbb{F}}$ and the 
isomorphism is induced by 
$\pi \colon C_{R,\mathbb{F}} \to C_{R_1,\mathbb{F}};\ 
(a,x) \mapsto (a-\Delta(x),r(x))$. 
In particular, we have 
$xR(x)=r(x) R_1(r(x))+\Delta(x)^p-\Delta(x)$. 
Furthermore, we have $d_{R,m} \mid d_{R_1}$. 
\end{proposition}
\begin{proof}
By applying Lemmas \ref{q0} and \ref{q2} successively, 
the quotient $C_{R,\mathbb{F}}/U'_R$ is isomorphic to 
the curve $C_{R_1,\mathbb{F}}$ with 
some $R_1(x) \in \mathscr{A}_{\mathbb{F}}$, and 
we obtain
$\pi \colon C_{R,\mathbb{F}} \to C_{R_1,\mathbb{F}};\ 
(a,x) \mapsto (a-\Delta(x),r(x))$. 
By \eqref{ch}, we have $\Delta(0)=0$. 
Since $U_R$ is an 
$\mathbb{F}_p[\mathscr{H}]$-module, 
the subgroup $A:=\{(\alpha,0,0) \in Q_{R,m} \mid \alpha \in\mu_{d_{R,m}}\}$ normalizes $U'_R$. 
Hence $A$ acts on the quotient $C_{R_1,\mathbb{F}}$. 
We recall that $b^p-b=y R_1(y)$ is the defining equation of $C_{R_1,\mathbb{F}}$. 
Then $A \ni (\alpha,0,0)$ acts on $C_{R_1,\mathbb{F}}$ is given by $b \mapsto b+\Delta(x)-\Delta(\alpha^{-1} x),\ y=r(x) \mapsto r(\alpha^{-1}x)=\alpha^{-1} y$ through 
the morphism $\pi$ by \eqref{r0r}. 
By \cite[Theorems (4.1) and (13.3)]{GV} or \cite[Theorem 4.3.2]{BHMSSV}, we must have $\alpha \in \mu_{d_{R_1}}$.  Hence the last claim follows. 
\end{proof}
\begin{corollary}\label{rpc}
Let the assumption be as in Proposition \ref{rp}. 
We have $\Delta(x), R_1(x) \in \mathbb{F}_q[x]$. 
\end{corollary}
\begin{proof}
We use the same notation in Definition 
\ref{eled}. We consider the equality 
$xR(x)=r(x) R_1(r(x))+\Delta(x)^p-\Delta(x)$ in 
Proposition \ref{rp}. Let 
$S(x):=-R_1^{\sigma}(x)+R_1(x)$ and 
$\Pi(x):=\Delta^{\sigma}(x)-\Delta(x)$.
We have $S(x) \in \mathscr{A}_{\mathbb{F}}$.  
By $r(x), R(x) \in \mathbb{F}_q[x]$, 
\begin{equation}\label{Pi}
\Pi(x)^p-\Pi(x)=r(x) S(r(x)). 
\end{equation}
Assume $S(x) \neq 0$. 
We have the non-constant morphism 
$f \colon \mathbb{A}_{\mathbb{F}}^1 \to C_{S,\mathbb{F}};\ 
x \mapsto (\Pi(x),r(x))$ by $\deg r(x)>0$.
Let $\overline{C}_{S,\mathbb{F}}$
be the smooth compactification of $C_{S,\mathbb{F}}$. The morphism $f$ extends to a non-constant morphism 
$\mathbb{P}_{\mathbb{F}}^1
\to \overline{C}_{S,\mathbb{F}}$. Hence this is a finite morphism. 
By the Riemann--Hurwitz formula, we know that 
the genus of $\overline{C}_{S,\mathbb{F}}$
equals zero. 
This is a contradiction by Lemma \ref{super}. 
Hence $S(x) \equiv 0$ and $R_1(x) \in \mathbb{F}_q[x]$. 
We have $\Pi(x) \in \mathbb{F}_p$ by
\eqref{Pi}.  
We have $\Pi(0)=0$ by
$\Delta(0)=0$ as in Proposition \ref{rp}. 
Hence $\Pi(x) \equiv 0$. Thus the claim follows. 
\end{proof}

 \subsection{Theorem}
 Finally, we summarize the contents of  
\S \ref{4.1} and \S \ref{4.2} as a 
theorem. 
  \begin{theorem}\label{4mm}
 Assume $p\neq 2$. 
 The following conditions are equivalent.
 \begin{itemize}
\item[{\rm (1)}]
There exists a non-trivial finite \'etale morphism 
\[
C_R \to C_{R_1};\ (a,x) \mapsto (a-\Delta(x),r(x)), 
\]
where $\Delta(x) \in \mathbb{F}_q[x]$ and $r(x), R_1(x) \in \mathscr{A}_q$ satisfy 
$d_{R,m} \mid d_{R_1}$ and 
$r(\alpha x)=\alpha r(x)$ for $\alpha \in 
\mu_{d_{R,m}}$. 
 \item[{\rm (2)}] The $\mathbb{F}_p[\mathscr{H}]$-module $(V_R,\omega_R)$ is not completely anisotropic. 
 \item[{\rm (3)}] The $W_F$-representation 
 $\tau_{\psi,R,m}$ is imprimitive. 
 \end{itemize} 
 If the above equivalent conditions are satisfied,  the $W_F$-representation 
 $\tau_{\psi,R,m}$ is isomorphic to 
 $\Ind_{W_{F'}}^{W_F} \tau'_{\psi, R_1, m}$, where 
 $\tau'_{\psi, R_1, m}$ is given in \eqref{taud}. 
 \end{theorem}
\begin{proof}
Assume (1). 
Since $C_R \to C_{R_1}$ is non-trivial, we have 
$e'<e$, where $\deg r(x)=p^{e-e'}$. 
Then we have (2) by Lemma \ref{tiso}. 
Assume (2). We obtain (1) by \eqref{r0r}, Proposition \ref{rp} and  Corollary \ref{rpc}.

The equivalence of (2) and (3) 
follows from Corollary \ref{mc}(1). 

The last claim follows from 
Corollary \ref{mc4}. 
\end{proof}
\subsubsection*{Acknowledgements}
This work was supported by JSPS KAKENHI Grant Numbers 20K03529/21H00973.

Takahiro Tsushima\\  
Department of Mathematics and Informatics, 
Faculty of Science, Chiba University
1-33 Yayoi-cho, Inage, 
Chiba, 263-8522, Japan \\
tsushima@math.s.chiba-u.ac.jp
\end{document}